%% file: NBWKAM-arxiv.tex
\documentclass[a4paper,10pt]{article}
\usepackage[utf8]{inputenc}
\usepackage{amssymb}
\usepackage{amsthm}
\usepackage{epsf} 

\usepackage{graphicx,psfrag}
\usepackage{xcolor,rotating,epic,eepic}

\newtheorem{theorem}{Theorem}
\newtheorem{lemma}[theorem]{Lemma}
\newtheorem{corollary}[theorem]{Corollary}
\newtheorem{proposition}[theorem]{Proposition}

\def\N{\mathbb{N}}

\def\R{\mathbb{R}}

\def\set#1{\left\{\, #1 \,\right\}}
\def\tq{\,\mid\,}
\def\abs #1{\vert \,#1\, \vert\,}
\def\norm #1{\Vert \,#1\, \Vert\,}

\def\R{\mathbb{R}}

\def\Leg{\mathcal{L}}
\def\calH{\mathcal{H}}
\def\calC{\mathcal{C}}

\def\calE{\mathcal{E}}

\begin{document}

\title{On weak KAM theory\\ for N-body problems}

\author{Ezequiel Maderna}

\date{July 27, $2006$}

\maketitle

\begin{abstract}
We consider N-body problems with homogeneous potential $1/r^{2\kappa}$ where $\kappa\in(0,1)$, including
the Newtonian case ($\kappa=1/2$). Given $R>0$ and $T>0$, we find a uniform upper bound for the minimal
action of paths binding in time $T$ any two configurations which are contained in some ball of radius $R$.
Using cluster partitions, we obtain from these estimates H\"{o}lder regularity of the critical action potential
(i.e. of the minimal action of paths binding in free time two configurations).
As an application, we establish the weak KAM
theorem for these N-body problems, i.e. we prove the existence of fixed points of
the Lax-Oleinik semigroup and we show that they are global viscosity solutions of the corresponding
Hamilton-Jacobi equation. We also prove that there are
invariant solutions for the action of isometries on the
configuration space.
\end{abstract}

\section{Introduction}

Let $E$ be a finite dimensional Euclidian space, and denote by
$x=(r_1,\dots,r_N)\in E^N$ the configuration vector of $N$
punctual masses $m_1,\dots,m_N>0$. By $\norm{x}$ we will denote
the norm given by $\max\set{\norm{r_i}_E \tq 1\le i\le N}$, and
$\abs{x}$ will denote the norm induced by the mass scalar product
$$<x,y>=\sum_{i=1}^N m_i<r_i,s_i>_E$$ for $x=(r_1,\dots,r_N)$,
$y=(s_1,\dots,s_N)\in E^n$. As usual, we call $I(x)=\abs{x}^2$ the
moment of inertia of $x$ regarding the origin of $E$. The
$N$-body problem is determined once the force function $U$ on $E^N$
(or potential function), negative of the potential energy, is
chosen. In this paper, we restrict us to the potential functions
which are homogeneous of degree $-2\kappa$
$$U_\kappa(x)=\sum_{i<j}m_im_j\,(r_{ij})^{\,-2\kappa}\,,$$ where
$r_{ij}=\norm{r_i-r_j}_E$, and $\kappa\in (0,1)$. The case
$\kappa=1/2$ corresponds to the Newtonian potential. In other words,
this means that the laws of motion are given on the open and dense
subset $\Omega = \set{x\in E^N\tq U_\kappa(x)<+\infty}$ by the
differential equation $\ddot x=\nabla U_\kappa$, where the gradient
is taken with respect to the mass scalar product on $E^N$. The
equivalent variational formulation is given by the Lagrangian
defined on $TE^N=E^N\times E^N$,
$$L(x,v)=\frac{1}{2}\sum_{i=1}^N m_i v_i^2+U_\kappa(x)\,,$$ where
$v=(v_1,\dots,v_n)$. Thus, motions are characterized as critical
points of the Lagrangian action $A(\gamma)=\int
L(\gamma(s),\dot\gamma(s))\,ds$, and the Euler-Lagrange equations
define a - non complete - analytical flow on the non compact
manifold $T\Omega$.

\subsection{Globally minimizing curves and the action potential.}

Let us give a precise definition of the Lagrangian action
functional. Recall that a curve $\gamma:[a,b]\to E^N$ is absolutely
continuous if it is differentiable almost everywhere, and its
derivative $\dot \gamma$ satisfies the fundamental theorem of
calculus for the Lebesgue integral. Thus the Lagrangian action is
well defined on the set of absolutely continuous curves $\calC$.
More precisely, the action is the function $A:\calC\to (0,+\infty]$
given by
$$A(\gamma)=\int_a^bL(\gamma(s),\dot\gamma(s))\,ds=
\frac{1}{2}\int_a^b\abs{\dot\gamma(s)}^2\,ds
+\int_a^bU_\kappa(\gamma(s))\,ds\,.$$ where $\abs{v}$ is the norm in
$E^N$ induced by the mass scalar product. It can be seen that
absolutely continuous curves with finite action are necessarily
$1/2$-H\"{o}lder continuous, hence they are contained in the Sobolev
space $H^1([a,b],E^N)$.

For $T>0$ and $x,y\in E^N$, denote by $\calC(x,y,T)$ the set of all
absolutely continuous curves $\gamma:[0,T]\to E^N$ which satisfy
$\gamma(0)=x$ and $\gamma(T)=y$. We are interested in the function
$\phi$ defined on $E^N\times E^N\times (0,+\infty)$ by
$$\phi(x,y,T)=\inf\set{A(\gamma)\tq \gamma\in\calC (x,y,T)}\,.$$ We
will say that a curve $\gamma:[a,b]\to E^N$ is globally minimizing,
if we have that $A(\gamma)=\phi(\gamma(a),\gamma(b),b-a)$. For a
curve defined on a non compact interval, globally minimizing will
mean that the property is satisfied for all restrictions of the
curve to a compact interval. It is not difficult to see that a
globally minimizing curve always exists for any two configurations
$x,y\in E^N$ and for all $T>0$. Essentially, it is a consequence of
the lower semi-continuity of the action functional.

In the last years, the global variational methods have been
successful to prove the existence of a great variety of particular
motions. A typical example is the eight choreography of Chenciner
and Montgomery \cite{Ch-Montgomery}, among many others closed orbits
with topological or symmetry constraints. The main difficulty that
raises from these methods for the Newtonian potential, and also for
the homogeneous potentials here considered, is the one to assure
that global minimizers avoid collisions, that is to say, that they
are contained in the open domain $\Omega$. Following an idea of
Marchal, Chenciner established a proof of this fact, for the
Newtonian N-body problem in the plane or the three-dimensional
space, see \cite{Ch2-ICM}, \cite{Marchal}. Simultaneously and independently,
Ferrario and Terracini gave an improved version of the Marchal's theorem, see \cite{Ferrario-Terracini}.
We will nowhere use this result in this paper, but it is relevant to remark that combined with
proposition \ref{fixpoints=calib} below, and using results from \cite{DaLuz-Maderna},
we can deduce (for the Newtonian case in dimension greater than one) the existence of completely
parabolic motions with arbitrary initial configuration. Recently, this last result was improved
in \cite{Maderna-Venturelli}.

Our first result gives an upper bound for the action of such curves
which depends on the size of the configurations. In our opinion,
this result is quite fundamental for global variational methods, and
it is optimal, in the sense that the bound is reached by homothetic
minimizing configurations, as we explain in the following section.

\begin{theorem}\label{upper.bound} There are positive constants
$\alpha,\;\beta\,>0$ such that for all $T>0$,
$$\phi(x,y,T)\leq \alpha\,T^{-1}R^2+\beta\,TR^{-2\kappa}\,,$$
whenever $x$ and $y$ are configurations contained in a ball of
radius $R>0$ of $E$. The constants $\alpha$ and $\beta$ only depend
on the degree of homogeneity of the potential ($-2\kappa$), the
number of bodies $N$, and their masses.
\end{theorem}

The next result shall be useful for the study of free time
minimizers, that is to say, absolutely continuous curves which
minimizes the action in the set of curves
$\calC(x,y)=\bigcup_{T>0}\calC(x,y,T)\,.$ The Ma\~{n}\'{e}'s critical action
potential (see for instance \cite{CIPP}), or the \textsl{action
potential}, is defined in our setting on $E^N\times E^N$ by
$$\phi(x,y)=\inf\set{\phi(x,y,T)\tq T>0}=\inf\set{A(\gamma)
\tq \gamma\in\calC(x,y)}\,.$$ It is clear that
$\phi(x,y)=\phi(y,x)$, and that $\phi(x,y)\leq\phi(x,z)+\phi(z,y)$,
for any configurations $x,\,y,\,z$ in $E^N$. In fact, proposition
\ref{phi.dist} shows that the action potential $\phi$ is a distance
function. Notice that as a corollary of theorem \ref{upper.bound},
we have that $\phi(x,y)\leq (\alpha+\beta)R^{1-\kappa}$ whenever $x$
and $y$ are configurations contained in a ball of radius $R>0$ of
$E$. With similar arguments as in theorem \ref{upper.bound},
combined with a cluster decomposition, we obtain the following
theorem.

\begin{theorem}\label{potencial.holder} There is a positive
constant $\eta>0$ such that for all $x,y\in E^N$,
$$\phi(x,y)\leq \eta\,\norm{x-y}^{1-\kappa}\,.$$
\end{theorem}

Therefore, the action potential is H\"{o}lder continuous respect to the
Euclidean norm on $E^N\times E^N$. In other words, for any
configurations $x,\,y,\,z$ in $E^N$ we have $\phi(x,z)-\phi(y,z)\leq
\phi(x,y)\leq\eta\,\norm{x-y}^{1-\kappa}$. On the other hand, it is
easy to prove that the action potential is locally Lipschitz in the
open and dense subset $\Omega\times\Omega\subset E^N\times E^N$.

The action potential $\phi$ was introduced by Ma\~{n}\'{e} in the nineties, as
well as the Fathi's weak KAM theorem, for the study of the dynamics of
Tonelli Lagrangians on compact manifolds. But in fact, for the Newtonian
case of our N-body problems, the corresponding action potential is nothing but
the minimal action between two configurations for the Maupertuis action functional.
This last can be defined as the energy functional associated to the Jacobi metric in
the zero energy level.

\subsection{On the weak KAM theory.}

In order to give applications, we will show that theorem
\ref{potencial.holder} enables us to prove a weak KAM theorem in the
spirit of \cite{Fathi-CRAS1}, \cite{Fathi-Maderna}. The novelty in
this viewpoint, is that we regard the action of the Lax-Oleinik
semigroup on a space of H\"{o}lder functions.

Let us remember that a function $u:E^N\to \R$ is said
\textsl{dominated} by $L$, if it satisfies the condition
$u(x)-u(y)\leq \phi(x,y)$ for all $x,y\in E^N$. Since the action
potential is symmetric, theorem \ref{potencial.holder} implies that
dominated functions are H\"{o}lder continuous. On the other hand, it is
not difficult to prove that they are locally Lipschitz in the open
subset of total measure $\Omega\subset E^N$, see proposition
\ref{phi.loc.lip} below. Therefore, dominated functions are
differentiable almost everywhere. We shall discuss this in more
detail below. Another way to define the set of dominated functions,
is using the Lax-Oleinik semigroup: given a function $u:E^N\to
[-\infty,+\infty)$ and $t>0$ we define
$T^-_tu:E^N\to[-\infty,+\infty)$ by
$$T^-_tu(x)=\inf\set{u(y)+\phi(x,y,t)\tq y\in E^N }\,.$$ Then,
a continuous function $u$ is dominated if and only if $u\leq T^-_tu$
for all $t>0$. Notice that the set of dominated functions is convex
and stable under the Lax-Oleinik semigroup. Setting $T^-_0u=u$ for
any function $u$, we will prove that $(T^-_t)_{t\geq 0}$ is a
continuous semigroup on the set of dominated functions equipped with
the topology of uniform convergence on compact subsets.

Another set which is stable by the Lax-Oleinik semigroup is the set
of functions which are invariant by symmetries. If we observe that
the group of isometries of $E$, acts naturally on $E^N$ by
symmetries of the potential function, then an obvious question is
the existence of invariant fixed points of the semigroup. More
precisely, we will say that a function $u:E^N\to\R$ is
\textsl{invariant} if $u(r_1, \dots , r_N)=u(Ar_1+r, \dots, Ar_N+r)$
for all $x=(r_1,\dots,r_N)\in E^N$, $r\in E$ and $A\in O(E)$.

\begin{theorem}[invariant weak KAM]\label{wkam.thm}
There exists an invariant and dominated function $u:E^N\to\R$ such
that $u=T^-_tu$ for all $t\geq 0$.
\end{theorem}

In section \ref{wkam}, we prove the weak KAM theorem, and we study the relationship with the Hamilton-Jacobi equation. More precisely, we show that weak KAM solutions are global viscosity solutions in $\Omega$.

An important difference with the compact case is that here the Aubry set is empty. In particular the technique used in \cite{Maderna} to prove the invariance of all solutions is not available.
Moreover, we will exhibit non invariant solutions for the Kepler problem in the plane, which is the subject of the last section.

\section{H\"{o}lder regularity of the action potential}\label{action}

This section is devoted to the study of the action potential, and to
give the proofs for theorems \ref{upper.bound} and
\ref{potencial.holder}.

\subsection{Proof of theorem \ref{upper.bound}.} Given $r\in E$ and $R>0$, we say that a
configuration $x=(r_1,\dots,r_N)\in E^N$ is contained in the ball
$B(r,R)$ when we have $\norm{r_i-r}_E\leq R$ for all $i=1,\dots,N$.
Suppose now that we have two configurations $x$ and $y$ such that
for some $r\in E$ and some $R>0$, both $x$ and $y$ are contained in
$B(r,R)$. If we tried to bound $\phi(x,y,T)$ with the action of a
linear path, then two problems arise. The first one is that the
linear path can present collisions in which case the action is
infinite. The second one is that, even if the linear path avoid
collisions, the distance between two given bodies can be arbitrary
small for both configurations, hence the action can be arbitrary
large. Both problems are solved in the following way: fix an
intermediate configuration $p$ with sufficiently large mutual
distances, and take the linear path from $x$ to $p$ defined on
$[0,T/2]$ followed by the linear path from $p$ to $y$ defined on
$[T/2,T]$. This path has no more than $2N(N-1)$ collisions, and we
can determine the values of $t\in[0,T]$ in which these collisions
happen. Thus, reparametrizing the path in such a way that in the new
times of collisions the action integral converges, we obtain the
following proposition, from which we can easily deduce theorem
\ref{upper.bound}.

\begin{proposition}\label{curva.min}
Given two configurations $x,\, y\in E^N$ contained in a ball
$B(r,R)$, $r\in E$, $R>0$, and given $T>0$, there is a curve
$\gamma\in\calC(x,y,T)$, such that $\gamma(t)$ is contained in
$B(r,6NR)$ for all $t\in[0,T]$,
$$\frac{1}{2}\int_0^T\abs{\dot\gamma(t)}^2\,dt\leq
\alpha\,T^{-1}R^2\,,\textit{ and }
\int_0^TU_\kappa(\gamma(t))\,dt\leq\beta\,TR^{-2\kappa}\,,$$ where
$\alpha$ and $\beta$ are positive constants
that only depend on the number of bodies, the total mass and the
degree of homogeneity of the potential function.
In fact we can take
$$\alpha=640\,\frac{1+\kappa}{1-\kappa}\,M\,N^4\; \textrm{ and
}\;\beta=2\,\frac{1+\kappa}{1-\kappa}\,N^{(4\kappa+2)}M^2\,.$$
\end{proposition}

\begin{proof} We first observe that it suffices to give the proof
for a fixed value of $T>0\,$: for $S>0$, we can define
$\sigma:[0,S]\to E^N$ as $\sigma(s)=\gamma(sT/S)$, and we have
$$\int_0^S\abs{\dot\sigma(s)}^2\,ds=
T^2S^{-2}\int_0^S\abs{\dot\gamma(sT/S)}^2\,ds=
S^{-1}T\int_0^T\abs{\dot\gamma(t)}^2\,dt\leq 2\alpha\,S^{-1}R^2\,,$$
$$\int_0^SU_\kappa(\sigma(s))\,ds=\int_0^SU_\kappa(\gamma(sT/S))\,ds=
ST^{-1}\int_0^TU_\kappa(\gamma(t))\,dt\leq\beta\,SR^{-2\kappa}\,.$$
We will then give the proof for $T=2$. Take $v\in E$ such that
$\norm{v}_E=6R$, and define $p=(p_1,\dots,p_N)\in E^N$ by
$$p_i=r+(i-1)\,v\;,\;\;\;i=1,\dots,N\,.$$
Therefore, the configuration $p$ is clearly contained in
$B(r,6(N-1)R)$. Notice also that the mutual distances
$p_{ij}=\norm{p_i-p_j}_E$ of $p$ are greater than $6R$ and smaller
than $6(N-1)R$. 

Let now $x=(r_1,\dots,r_N)$ be a configuration such that
$\norm{r_i-r}_E\leq R$ for all $i=1,\dots,N$. We consider the curve
$z_x:[0,1]\to E^N$, defined by $z_x(t)=x+\psi_x(t)(p-x)$, where
$\psi_x:[0,1]\to[0,1]$ is an increasing function, with $\psi_x(0)=0$
and $\psi_x(1)=1$, to be determined. Our aim is to
choose the function  $\psi_x$ conveniently, in order to obtain a
bound of $A(z_x)$ which does not depend on $x$.

Recall that if  $u$ and $v$ are two vectors in a Euclidean space,
and $v\neq 0$, then we have, for all real number $\lambda$,
$$\norm{u+\lambda
v}^2=\left(\lambda\,\norm{v}+\frac{<u,v>}{\norm{v}}\right)^2+
\norm{u}^2-\frac{<u,v>^2}{\norm{v}^2}\,.$$ As a consequence, we get
$$\norm{u+\lambda v}\geq
\norm{v}\,\left|\,\lambda+\frac{<u,v>}{\norm{v}^2}\,\right|\,,$$ and
the minimum of $\norm{u+\lambda v}$ is reached for
$$\lambda=-\frac{<u,v>}{\norm{v}^2}\,.$$

We will use the notation $u_{ij}=r_i-r_j$ and
$v_{ij}=(p_i-p_j)-(r_i-r_j)$ for $i<j$. Thus, the mutual distances
of the configuration $z_x(t)$ can be written
$d_{ij}(t)=\norm{u_{ij}+\psi_x(t)v_{ij}}$.
Therefore, taking $\lambda=\psi_x(t)$, $u=u_{ij}$ and $v=v_{ij}$ in
the above considerations, we deduce that each mutual distance
$d_{ij}(t)$ verifies
$$d_{ij}(t)\geq \norm{v_{ij}}_E\,\abs{\psi_x(t)-t_{ij}}\geq
4R\,\abs{\psi_x(t)-t_{ij}}\,,$$ where
$$t_{ij}=-\frac{<u_{ij},v_{ij}>_E}{\norm{v_{ij}}_E^2}\,.$$
Since $\norm{u_{ij}}_E\leq 2R$ and $\norm{v_{ij}}_E\geq 4R$ for all $i<j$, we obtain
that $\abs{t_{ij}}<1/2$ for all $i<j$.
Therefore, since then number of pairs $(i,j)$
with $1\leq i<j\leq N$ is bounded by $N^2/2$,
assuming lemma \ref{reparametrisation} below, we can choose the function
$\psi_x$ can be chosen in such a way that, on one side,
$$\int_0^1\dot\psi_x(t)^2dt\leq 5N^2\frac{1+\kappa}{1-\kappa}\,,$$
and on the other side, for each $i<j$ there is a real number
$s_{ij}$ for which
$$\abs{\psi_x(t)-t_{ij}}\geq N^{-2}\abs{t-s_{ij}}^{(1/1+\kappa)}$$ for all
$t\in[0,1]$. Let us estimate the action $A(z_x)$ for this function
$\psi_x$. We have $\dot z_x(t)=\dot\psi_x(t)\,(p-x)$, and
$\norm{p_i-r_i}_E\leq 8NR$, for all $i=1,\dots,N$. Hence, by the previous estimates
we deduce that

\begin{eqnarray*}
\frac{1}{2}\int_0^1\abs{\dot
z_x(t)}^2\,dt&=&\frac{1}{2}\sum_{i=1}^Nm_i\norm{p_i-r_i}_E^2
\int_0^1\dot\psi_x(t)^2dt\\&&\\ &\leq&
160\,\frac{1+\kappa}{1-\kappa}\,M\,N^4\,R^2\,,
\end{eqnarray*}
and that

\begin{eqnarray*}
\int_0^1 U_\kappa(z_x(t))\,dt&=&
\sum_{i<j}\int_0^1m_im_j\,d_{ij}(t)^{-2\kappa}\,dt\\&&\\&\leq&\;
\sum_{i<j}\;\int_0^1m_im_j\,(4R)^{-2\kappa}\,N^{4\kappa}\,
\abs{t-s_{ij}}^{-(2\kappa/1+\kappa)}\,dt\,.
\end{eqnarray*}
Using that for $r\in (0,1)$, and for any $s\in\R$, we have
$$\int_0^1 \frac{1}{\abs{t-s}^r}\,dt=
\int_{-s}^{1-s} \frac{1}{\abs{u}^r}\,du
\leq 2\,\int_0^1 \frac{1}{u^r}\,du\,=2\,(1-r)^{-1}\;,$$
we deduce that
$$\int_0^1 U_\kappa(z_x(t))\,dt\;
\leq\;2\,\frac{1+\kappa}{1-\kappa}\,M^2\,N^{(4\kappa+2)} R^{-2\kappa}\,.$$

To finish the proof, let $y=(s_1,\dots,s_N)$ be a second
configuration contained in $B(r,R)$, and define
$\gamma\in\calC(x,y,2)$ as follows: $\gamma(t)=z_x(t)$ if $t\leq 1$,
and $\gamma(t)=z_y(2-t)$ if $t\geq 1$. We conclude that
$$A(\gamma)=A(z_x)+A(z_y)\leq\,320\,\frac{1+\kappa}{1-\kappa}\,M\,N^4\,R^2\;+
\;4\,\frac{1+\kappa}{1-\kappa}\,N^{(4\kappa+2)}M^2\,R^{-2\kappa}\,.$$ This also
proves the proposition for $T=2$, with
$$\alpha=640\,\frac{1+\kappa}{1-\kappa}\,M\,N^4\; \textrm{ and
}\;\beta=2\,\frac{1+\kappa}{1-\kappa}\,N^{(4\kappa+2)}M^2\,.$$
It is important to note that for the Newtonian case ($\kappa=1/2$), the dependence of both constants
in the number of bodies is in $N^4$.
\end{proof}\medbreak

We have used the following lemma.

\begin{lemma}\label{reparametrisation}
Given $\kappa\in (0,1)$ and real numbers $a_1<\dots<a_m$, there
are real numbers $b_1<\dots<b_m$ and an increasing absolutely continuous homeomorphism $F$ of $[0,1]$ such that
\begin{enumerate}
\item $$\abs{F(t)-a_i}\geq \frac{1}{2m}\abs{t-b_i}^{1/1+\kappa}$$ for all $t\in[0,1]$ and each $i=1,\dots,m$, and
\item $$\int_0^{\,1}F'(t)^2dt\leq (4+2a)(m+1)\frac{1+\kappa}{1-\kappa}\,,$$ where $a=\min\set{\abs{a_1},\dots,\abs{a_m}}$.
\end{enumerate}
\end{lemma}
\begin{proof}
Given $c>0$, let $g_c:\R\setminus\set{0}\to\R$ defined by
 $g_c(x)=c\,\abs{x}^{-\kappa/1+\kappa}$. Given $b=(b_1,\dots,b_m)\in\R^m$
 such that $b_1<\dots<b_m$, we also define the function
 $f_{b,c}:\R\setminus\set{b_1,\dots,b_m}\to\R$ by
 $$f_{b,c}(t)=\max\set{g_c(t-b_1),\dots,g_c(t-b_m)}\,.$$

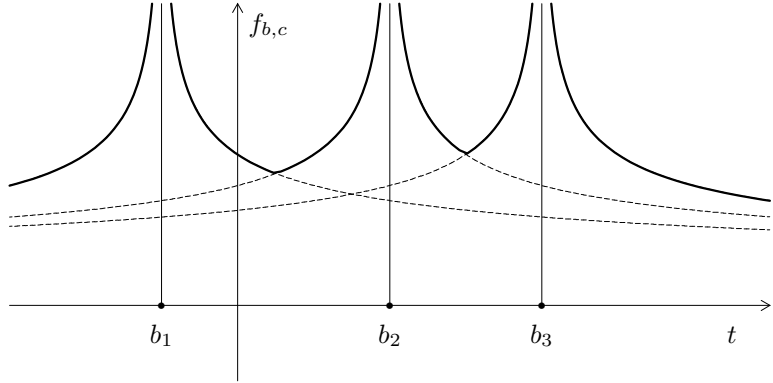
\begin{figure}[ht]
\center
\input{speedrepar.tex}
\caption{Graph of $f_{b,c}$ for $b=(b_1,b_2,b_3)$.}\label{speed}
\end{figure}

We will define the required function $F$ as a primitive of a
function $f_{b,c}$ for a good choice of $b$ and $c$. More precisely,
we define $F:\R\to\R$ by $$F(t)=\int_0^tf_{b,c}(s)\,ds\,.$$ The map $F$ is an increasing
homeomorphism of $\R$. As any primitive, $F$ is absolutely continuous. Moreover, we have
$$\abs{F(t)-F(b_i)}\geq c\,(1+\kappa)\,\abs{t-b_i}^{1/1+\kappa}$$
for all $t\in\R$ and each $i=1,\dots,m$. Therefore, we must choose
$c>0$ and $b=(b_1,\dots,b_m)\in\R^m$ such that $F(1)=1$ and
$F(b_i)=a_i$ for all $i=1,\dots,m$.

\begin{figure}[ht]
\center
\input{repar.tex}
\caption{Graph of $F(t)=\int_0^t f_{b,c}(s)\,ds$.}\label{repar}
\end{figure}
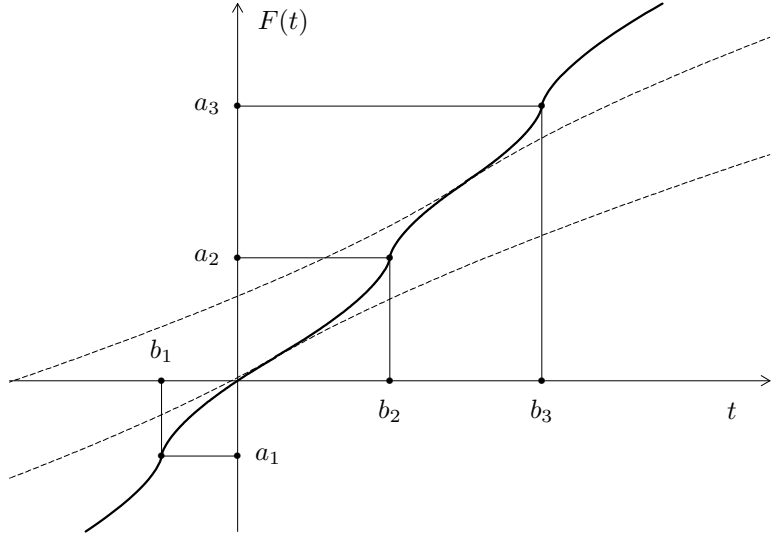

If we fix $c>0$, then we have a unique possible choice for $b$. To see this, first observe that the $m-1$ distances between the consecutive values of $a_i$ determine the $m-1$ distances between the consecutive values of $b_i$. If we set $A_i=a_{i+1}-a_i$ and $B_i=b_{i+1}-b_i$ then we must have $$A_i=\int_{b_i}^{b_{i+1}}f_{b,c}(s)\,ds=
2c\int_0^{B_i/2}s^{-\kappa/1+\kappa}ds=
2^{\kappa/1+\kappa}\,c\,(1+\kappa)\,B_i^{1/1+\kappa}$$ hence
$B_i=2^{-\kappa}[A_i/c\,(1+\kappa)]^{1+\kappa}$. From the condition $F(b_1)=a_1$ we deduce that
$$a_1=\int_0^{b_1}f_{b,c}(s)\,ds=-\int_0^{-b_1}f_{b,c}(s+b_1)\,ds\,.$$
Therefore $b_1$ must be the unique solution of the equation $\int_0^{-x}f_{b',c}(s)\,ds=-a_1$ where $b'=(0,B_1,B_1+B_2,\dots,B_1+\dots+B_{m-1})=
(0,b_2-b_1,\dots,b_m-b_1)$.
Moreover, we have showed that there is a continuous vector $b(c)\in\R^m$ such that $\int_0^{b_i}f_{b(c),c}(s)\,ds=a_i$ for all $i=1,\dots,m$. Therefore, it is clear that $$\delta(c)=\int_0^1f_{b(c),c}(s)\,ds$$ also depends continuously on $c$.
We claim that there is $c\in[1/2m(1+\kappa),2+a]$ for which $\delta(c)=1$.
We have
$$\delta(c)=\int_0^1f_{b(c),c}(s)\,ds \leq \sum_{i=1}^m\int_0^1g_c(s-b_i(c))\,ds\,.$$
Since $$\int_0^1g_c(s-b_i(c))\,ds\,
=\int_{-b_i(c)}^{1-b_i(c)}c\,\abs{u}^{-\kappa/1+\kappa}du
\leq 2c\,\int_0^1u^{-\kappa/1+\kappa}du=2c\,(1+\kappa)\,,$$
we deduce that $\delta(c)<1$ when $c<1/2m(1+\kappa)$. In order to prove the claim, it suffices to show that $\delta(c)>1$ when $c>2+a$.
Since $a=\abs{a_j}$ for some $j\in\set{1,\dots,m}$, we have
\begin{eqnarray*}
a&=&\abs{\int_0^{b_j(c)}f_{b(c),c}(s)\,ds}\\
&\geq&
\abs{\int_0^{b_j(c)}c\,\abs{s-b_j(c)}^{-\kappa/1+\kappa}\,ds}\\
&\geq&
c\,\int_0^{\abs{b_j(c)}}s^{-\kappa/1+\kappa}\,ds=c\,(1+\kappa)
\abs{b_j(c)}^{1/1+\kappa}\,,
\end{eqnarray*}
which implies $\abs{b_j(c)}\leq [a/c(1+\kappa)]^{(1+\kappa)}$.
On the other hand we have
\begin{eqnarray*}
\delta(c)&\geq &\min\set{g_c(s-b_j(c))\mid s\in[0,1]}\\
&\geq&c\,(1+\abs{b_j(c)})^{-\kappa/1+\kappa}.
\end{eqnarray*}
Thus, it suffices to prove that $\abs{b_j(c)}\leq c^{(1+\kappa)/\kappa}-1$ when $c>2+a$. By the previous estimation of $\abs{b_j(c)}$, we only have to prove that $(a/1+\kappa)^{1+\kappa}\leq c^{1+\kappa}(c^{1+\kappa/\kappa}-1)$, but this condition is clearly satisfied if $c>2$ and $c>a$.

We take $c\in [1/2m(1+\kappa),2+a]$ such that $\delta(c)=1$ and we define $F:[0,1]\to[0,1]$ by $$F(t)=\int_0^tf_{b(c),c}(s)\,ds\,.$$ In order to see that this function satisfy all the required conditions, it remains to estimate the $L^2$ norm of $F'$.
If we observe that $[0,1]\setminus\set{b_1(c),\dots,b_m(c)}$ has at most $m+1$ components $I_j$, and that on each one of these components we have $$\int_{I_j}f_{b(c),c}(s)^2\,ds\leq 2\int_0^1 c\,s^{-2\kappa/1+\kappa}ds\,,$$ we conclude that $$\int_0^1 F'(t)^2dt\leq (4+2a)(m+1)\frac{1+\kappa}{1-\kappa}\,.$$
\end{proof}\medbreak

\subsection{Minimal configurations.}

The following observations show that theorem \ref{upper.bound} is
optimal in the sense that the bound is reached by some
configurations. We shall first recall the notions of
\textsl{central} and \textsl{minimal} configurations, as well as
some properties (see for instance Wintner \cite{Wintner}, where the Newtonian case is discussed).

We say that a configuration $x\in E^N$ is \textsl{minimal}, if it
is a minimum of the potential function $U_\kappa$ restricted to the sphere
$\set{y\in E^N\tq I(y)=I(x)}$. In particular, minimal
configurations are \textsl{central} configurations, that is, critical points of $\widetilde{U}_\kappa=I^\kappa\,U_\kappa$, or in other words,
configurations $x\in E^N$ which are critical points of $U_\kappa$
restricted to $\set{y\in E^N\tq I(y)=I(x)}$. Central configurations
are also characterized as configurations which admit homothetic motions.
In other words, a configuration $x_0\in E^N$ is central, if and only
if $U_\kappa(x_0)<+\infty$ and $x(t)=r(t)x_0$ is a solution of the N-body problem for some
positive real function $r(t)$.

Take $x_0\in E^N$ a central configuration. If we look for an homothetic motion
through $x_0$, then we must solve a one dimensional differential
equation satisfied by $r(t)$.
A particular solution, that we shall call parabolic, is given by $x(t)=c\,t^{1/1+\kappa}x_0$ for
some value of $c>0$. A simple computation shows that the action of this solution is
\begin{eqnarray*}
A(x\mid_{[0,T]})&=&\frac{c^2}{2(1+\kappa)^2}
\int_0^T\,t^{-2\kappa/1+\kappa}\,dt\;+\;c^{-2\kappa}U_\kappa(x_0)
\int_0^T\,t^{-2\kappa/1+\kappa}\,dt\\&&\\
&=&\left(\frac{c^2}{2(1-\kappa^2)}+
c^{-2\kappa}U_\kappa(x_0)\frac{1+\kappa}{1-\kappa}\right)\,T^{(1-\kappa)/(1+\kappa)}\,.
\end{eqnarray*}
If we set $R_T=\norm{x(T)}=T^{1/1+\kappa}\,\norm{c\,x_0}$, then we
can write
$$A(x\mid_{[0,T]})=\alpha_0\,T^{-1}R_T^2\,+\,\beta_0\,T\,R_T^{-2\kappa}\,,$$
for a good choice of constants $\alpha_0$ and $\beta_0$.

On the other hand, we will prove that if $x_0$ is a minimal configuration, then the
above solution $x(t)$ is globally minimizing. In other words, we
have
$$\phi(0,x(T),T)=A(x\mid_{[0,T]})\,,$$ for all $T>0$, therefore the bound for $\phi(x,y,T)$
given by theorem \ref{upper.bound} cannot be improved modulo the
choice of the constants. We will assume for simplicity that $x_0$ is also normal, meaning that $I(x_0)=1$.

In order to prove that $x$ is globally minimizing,
we will first study the homogeneous one center problem in dimension one
which is satisfied by the function $r(t)=I(x(t))^{1/2}$.
More precisely, we have that $r(t)$ must be an extremal for the
Lagrangian system in $\R^+=[0,+\infty)$ defined by
$$L_0(r,v)=\frac{\;v^2}{2}+ \frac{U_0}{\;r^{2\kappa}}\;,$$
where $U_0=U\kappa(x_0)$.

It can be proved easily using the lower semi-coninuity of the Lagrangian
action of $L_0$ that, given $0\leq r_1<r_2$,
there is at least one absolute minimizer between $r_1$ and $r_2$.
This means that there is a curve $\gamma(r_1,r_2):[0,T]\to \R^+$,
for some positive time $T>0$, such that $\gamma(r_1,r_2)(0)=r_1$,
$\gamma(r_1,r_2)(T)=r_2$, and such that $\gamma(r_1,r_2)$
minimizes the Lagrangian action in the set of all absolutely continuous
curves $\sigma:[a,b]\to \R^+$ with $\sigma(a)=r_1$, $\sigma(b)=r_2$, and $a<b$.
Moreover, the fact that this absolute minimization property holds
(i.e. with fixed extremities but in free time),
implies that the energy of the extremal $\gamma(r_1,r_2)$ must be critical (zero).
Therefore $\gamma(r_1,r_2)$ satisfies the differential equation
$\dot\gamma^2=2U_0\,\gamma^{-2\kappa}$ and we conclude the uniqueness of such
absolute minimizer. By integration we get that for any $r>0$,
the absolute minimizer $\gamma(0,r):[0,T]\to\R^+$ is defined for
$T=(1+\kappa)(2U_0)^{-1/2}r^{(1+\kappa)}$
and that
$\gamma(0,r)(t)=c\,t^{1/1+\kappa}$ where
$c=(2U_0)^{1/(2+2\kappa)}
(1+\kappa)^{1/1+\kappa}$.

We prove now that if $x_0$ is a minimal configuration,
then the parabolic motion $x(t)=c\,t^{1/1+\kappa}x_0$ defined above
is globally minimizing.
Fix $T>0$, and take any other curve $\gamma\in\calC(0,x(T),T)$.
We must to prove that $A(\gamma)\geq A(x\mid_{[0,T]})$.
In fact we will prove that if $S=\sup\set{t\in [0,T]\mid \gamma(t)=0}$
then $$A(\gamma\mid_{[S,T]})\geq A(x\mid_{[0,T]})\,.$$
Setting $\gamma_1=\gamma\mid_{[S,T]}$ we have
$\gamma_1(t)\neq 0$ for all $t\in (S,T]$. Thus we can set
$\gamma_1(t)=r(t)s(t)$, where
$r(t)=\abs{\gamma_1(t)}=I(\gamma_1(t))^{1/2}$ for $t>S$.
Obviously we have that
$\abs{s(t)}=1$ for all $t>S$, and the action of $\gamma_1$ can be
written
$$A(\gamma_1)=\frac{1}{2}\int_S^T \dot r(t)^2\,dt\,+\,
\frac{1}{2}\int_S^T r(t)^2\abs{\dot s(t)}^{\,2}\,dt\,+\, \int_S^T
r(t)^{-2\kappa}U_\kappa(s(t))\,dt\,.$$ Since $x_0$ is minimal, we have
$U(s(t))\geq U_0=U\kappa(x_0)$ for all $t>S$. Moreover, we have
$$A(\gamma_1)\geq A(r\,x_0)=\int_S^T \left(\frac{1}{2}\;\dot r(t)^2\,dt\,+\,
U_0 r(t)^{-2\kappa}\right)\,dt\,.$$
Note that the last integral is nothing but the Lagrangian action of the curve $r(t):[S,T]\to\R^+$ for the one dimensional homogeneous one center problem, and as we have see the absolute minima of the Lagrangian action is reached by the zero energy solution $r(t)=c\,t^{1/1+\kappa}$. But in that case, the last integral is also de Lagrangian action of the parabolic motion $x(t)$ restricted to the interval $[0,T]$ for the N-body problem.
We deduce that
$$A(\gamma)\geq A(\gamma_1)\geq \int_0^T \left(\frac{1}{2}\;\dot r(t)^2\,dt\,+\,
U_0 r(t)^{-2\kappa}\right)\,dt\,=A(x\mid_{[0,T]})$$
hence we conclude that the solution $x(t)$ is globally minimizing.

\subsection{Properties of the action potential and proof of theorem
\ref{potencial.holder}.}

We start this section by showing that the action potential is a distance function
on $E^N$.

\begin{proposition}
\label{phi.dist} For all $x,y\in E^N$ we have $\phi(x,y)=0$ if and
only if $x=y$.
\end{proposition}

\begin{proof} Let $x\in E^N$ be a configuration, and choose a path
$\sigma:[0,1]\to E^N$ which satisfies $\sigma(0)=x$ and
$A(\sigma)<+\infty$. Then define for $0<T\leq 2$ the curve
$\gamma_T\in\calC(x,x,T)$ by $\gamma_T(t)=\sigma(t)$ if $t\leq T/2$,
and $\gamma_T(t)=\sigma(T-t)$ if $t\geq T/2$. It is not difficult to
see that $A(\gamma_T)\to 0$ as $T\to 0$, from which it follows that
$\phi(x,x)=0$ for all $x\in E^N$.

To see that the condition is necessary, take any two configurations
$x=(r_1,\dots,r_N)$ and $y=(s_1,\dots,s_N)$ in $E^N$, and a path
$\gamma=(\gamma_1,\dots,\gamma_N)\in \calC(x,y)$. If $d=\norm{y-x}$,
and $\gamma$ is defined on $[0,T]$, then it must exist $T_0\in
[0,T]$ such that $\norm{\gamma(T_0)-x}=d$ and
$\norm{\gamma(t)-x}\leq d$ for all $t\in [0,T_0]$. Moreover, we must
have $d=\norm{\gamma_i(T_0)-r_i}_E$ for some $i\in\set{1,\dots,N}$.
If $T_0\geq 1$, we can write
$$A(\gamma)\geq A(\gamma\mid_{[0,T_0]})\geq \int_0^{T_0}
U_\kappa(\gamma(t))\,dt\geq C>0\,,$$ where
$C=\min\set{U_\kappa(z)\tq \norm{z-x}\leq d}$. If $T_0\leq 1$ we
have
$$A(\gamma)\geq A(\gamma\mid_{[0,T_0]})\geq \frac{m_i}{2}\int_0^{T_0}
\norm{\dot\gamma_i(t)}_E^2\,dt\geq\frac{m\,d^2}{2}\,,$$ where
$m=\min\set{m_1,\dots,m_N}$. The last inequality follows from the
fact that $\gamma_i$ is absolutely continuous and the
Cauchy-Schwartz inequality. Therefore, we conclude that if
$\phi(x,y)=0$, then $d=0$ and $x=y$. \end{proof}\medbreak

In the sequel we will denote $\delta(z)$ the minimal distance
between the bodies of the configuration $z$. More precisely,
$\delta:E^N\to\R^+$ will be the function defined by
$\delta(z)=\min\set{\norm{z_i-z_j}_E\tq i<j}$, where
$z=(z_1,\dots,z_N)$. Thus the set of configurations without
collisions is nothing but $\Omega=\set{z\in E^N\tq \delta(z)>0}$. From the following proposition we can easily deduce that the action potential is a locally Lipschitz function in $\Omega\times\Omega$.

\begin{proposition}\label{phi.loc.lip} There are continuous functions $k(z)>0$
and $\epsilon(z)>0$ in $\Omega$ such that, if $x,z \in E^N$ satisfy
$\norm{x}<\epsilon(z)$, then $\phi(z,z+x)\leq k(z)\norm{x}$.
\end{proposition}

\begin{proof} We give the proof for $\epsilon(z)=\delta(z)/4$. Since
$z$ is without collisions, we have $\epsilon(z)>0$. For $T>0$ we define
the curve $\gamma:[0,T]\to E^N$, by $\gamma(t)=z+(t/T)x$. If
$z=(z_1,\dots,z_N)$ and $x=(r_1,\dots,r_N)$, then
$\gamma(t)=(\gamma_1(t),\dots,\gamma_N(t))$, where
$\gamma_i(t)=z_i+(t/T)r_i$. Hence, for $i<j$ and $t\in[0,T]$ we can
write
$$\gamma_{ij}(t)=\norm{\gamma_i(t)-\gamma_j(t)}_E\geq
\norm{z_i-z_j}_E-(t/T)\norm{r_i-r_j}_E\geq \delta(z)/2\,,$$ and
$$U_\kappa(\gamma(t))=\sum_{i<j}m_im_j\,\gamma_{ij}(t)^{-2\kappa}
\leq M^2N^2\left[\delta(z)/2\right]^{-2\kappa}\,.$$ Therefore, using
that $\abs{x}^2=I(x)\leq MN\norm{x}^2$, we deduce that
\begin{eqnarray*}
A(\gamma)&=&\frac{1}{2}\int_0^T\abs{x/T}^2\,dt\,+\,\int_0^TU(\gamma(t))\,dt\\
&&\\&\leq&MN\,\norm{x}^2/2T\,+\,M^2N^2\,
\left[\delta(z)/2\right]^{-2\kappa}\,T\,.
\end{eqnarray*}
If $x=0$ there is nothing to prove, since we already know that
$\phi(z,z)=0$. If $x\neq 0$, we can take $T=\norm{x}$, and the above
estimation gives $A(\gamma)\leq k(z)\norm{x}$ for $k(z)=MN/2+\,M^2N^2\,
\left(\delta(z)/2\right)^{-2\kappa}$. \end{proof}\medbreak

We introduce now a notion of \textsl{cluster partition} of a subset
$A\subset E$ adapted to our purposes. Given $\lambda>1$, we will say
that the set $\set{r_1,\dots,r_K}\subset E$ defines a
$\lambda$-cluster partition of size $R>0$ of $A$, if the following
two conditions are satisfied:
\begin{enumerate}
\item $\norm{r_i-r_j}_E\geq 2\lambda R\;$ for all $1\leq i< j\leq K$,
\item $A$ is contained in the union $\bigcup_{i=1}^{\,K} B(r_i,R)$.
\end{enumerate}
It is clear that if $A$ is finite and $R$ is small enough, then $A$
defines itself a cluster partition of size $R$ of $A$. It is also
clear that if $A$ is bounded, then any $r\in A$ defines a trivial
cluster partition of size $R$ for any $R>\textrm{diam}(A)$.

We will need the following lemma.

\begin{lemma}\label{clusters} Given $\lambda>1$, $A=\set{r_1,\dots,r_N}\subset E$ and
$\epsilon>0$, there is a subset $A'\subset A$, and $R(\epsilon)>0$
such that: (i) $\epsilon\leq R(\epsilon)<(2\lambda)^N\,\epsilon$,
(ii) $A'$ defines a $\lambda$-cluster partition of size
$R(\epsilon)$ of $A$.
\end{lemma}

\begin{proof} We reason recursively. We begin setting $A'_1=A$. If
$A'_1$ does not define a $\lambda$-cluster partition of size
$\epsilon$, then there are $r,s \in A'_1$ such that $\norm{r-s}_E<
2\lambda\epsilon\;$. If that is the case, we define
$A'_2=A'_1\setminus\set{s}$. Then we reason as before: if $A'_2$
does not defines a $\lambda$-cluster partition of size
$2\lambda\,\epsilon$ then we have $r,s\,\in A'_2$ such that
$\norm{r-s}_E< (2\lambda)^2\,\epsilon\;$, and we set
$A'_3=A'_2\setminus\set{s}$. It is clear that the process finish at
the most in $N$ steps. \end{proof}\medbreak

Using the existence of cluster partitions we will prove the following proposition, from which theorem \ref{potencial.holder} can be deduced as a simple corollary.

\begin{proposition}\label{upper.bound.epsilon} There are positive constants
$\alpha_1,\;\beta_1\,>0$ such that, for all $x,y\in E^N$ and for all $T>0$ we have
$$\phi(x,y,T)\leq \alpha_1\,T^{-1}\epsilon^2+\beta_1\,T\epsilon^{-2\kappa}\,,$$
whenever $\epsilon>\norm{x-y}$. The constants $\alpha_1$ and $\beta_1$ only depend
on the degree of homogeneity of the potential ($-2\kappa$), the
number of bodies $N$, and their masses.
\end{proposition}

\begin{proof}
Fix a configuration
$x=(r_1,\dots,r_N)\in E^N$, and denote by $A_x$ the set
$\set{r_1,\dots,r_N}\subset E$. Let $y=(s_1,\dots,s_N)\in E^N$ be any other configuration. If we apply lemma \ref{clusters} to
$A_x$ with $\epsilon>\norm{y-x}$ and $\lambda=24N$, we conclude that
there are $r_{i_1}\,,\dots, r_{i_K}\in A_x$, and $R(\epsilon)>0$
with the following properties.
\begin{enumerate}
\item $\epsilon\leq R(\epsilon)< (48N)^N\,\epsilon$,
\item for all $1\leq j<k\leq K$, we have $\norm{r_{i_j}-r_{i_k}}_E\geq
48N\,R(\epsilon)$, and
\item $A_x\cup A_y$ is contained in the disjoint union
$\bigcup_{j=1}^{\,K}B_j$ where $B_j=B(\,r_{i_j}\,,2R(\epsilon))$.
\end{enumerate}
Therefore, both configurations $x$ and $y$ are decomposed in $K$
clusters, each one contained in a ball $B_j$. More precisely, we
have a partition $\set{1,\dots,N}=I_1\cup\dots\cup I_K$ such that
$i\in I_j$ if and only if both $r_i$ and $s_i$ are in $B_j\,$.
Denote by $N_j=\textrm{card}(I_j)$ the number of bodies in cluster
$j$, and by $M_j$ the total mass of this cluster, that is
$M_j=\sum_{\,i\in I_{j}}m_i$. Thus we have $N=N_1+\dots+N_K$ and
$M=M_1+\dots+M_K$.

We consider now the $N_j\,$-body problem composed by the bodies in
the ball $B_j\,$. Given $T>0$, we apply proposition \ref{curva.min}
in each ball $B_j\,$, $j=1,\dots,K$, with initial and final
condition conformed by the $N_j$ bodies of $x$ and $y$ contained in
$B_j\,$. Therefore we obtain, a path
$\gamma=(\gamma_1,\dots,\gamma_N)\in\calC(x,y,T)$ such that for all
$j=1,\dots,K$ we have,
\begin{enumerate}
\item If $i\in I_j\,$, then $\gamma_i(t)\in
B(r_{i_j},12N\,R(\epsilon))$ for all $t\in [0,T]$,
\item $$T_j=\frac{1}{2}\int_0^T\sum_{i\in
{I_j}}m_i\norm{\dot\gamma_i(t)}_E^2\,dt\leq\,
\,10^6\,\frac{1+\kappa}{1-\kappa}\,M_j\,N_j^6\;R(\epsilon)^2/\,T\,,\textrm{
and}$$
\item \begin{eqnarray*}
    W_j &=&\int_0^T\sum_{i,\,k\,\in I_j}^{i<k}m_im_k
    \norm{\gamma_i(t)-\gamma_k(t)}_E^{-2\kappa}\,dt\\
    &\leq&
    2\,\frac{1+\kappa}{1-\kappa}\,N_j^{2\kappa+2}M_j^2\,
    12^{-2\kappa}\,R(\epsilon)^{-2\kappa}\,T\,.
    \end{eqnarray*}
\end{enumerate}
Notice that the action of the curve
$\gamma=(\gamma_1,\dots,\gamma_N)\in\calC(x,y,T)$ is
$$A(\gamma)=\sum_{j=1}^K T_j\,+\,\sum_{j=1}^K W_j\,+W_0$$ where $W_0$
is the integral of the terms of the potential function $U_\kappa$
corresponding to pairs of bodies in different clusters. More
precisely,
$$W_0=\int_0^T\sum_{1\leq j<\,l\leq K}\;\;\sum_{i\in I_j,\,k\in I_l}
m_im_k\,\norm{\gamma_i(t)-\gamma_k(t)}_E^{-2\kappa}\,dt\,.$$ Since
the balls $B(r_{i_j},24N R(\epsilon))$ are disjoint, we deduce that
$$W_0\leq N^2M^2\,(24N)^{-2\kappa}\,R(\epsilon)^{-2\kappa}\,T\,.$$
Using that using that
$R(\epsilon)<(24N)^N\,\epsilon$, we can write $$A(\gamma)<\alpha_1\,T^{-1}\epsilon^2+\beta_1\,T\epsilon^{-2\kappa}\,,$$
for some positive constants $\alpha_1$ and $\beta_1$ only depending on $N$, $M$ and $\kappa$.
\end{proof}\medbreak

\begin{corollary}\label{bound.phixxt}
There is a positive constant $\mu>0$ such that for all $x\in E^N$ $$\phi(x,x,T)\leq \mu\, T^{(1-\kappa)/(1+\kappa)}\,.$$
\end{corollary}

\begin{proof} It suffices to take $\epsilon=T^{1/1+\kappa}$ in proposition \ref{upper.bound.epsilon}.
\end{proof}\medbreak

\begin{proof}[Proof of theorem \ref{potencial.holder}.]
By the previous corollary we know that $\phi(x,x)=0$ for all $x\in E^N$. If $x,y\in E^N$ are two different configurations, then proposition \ref{upper.bound.epsilon} says that
$$\phi(x,y,T)<\alpha_1\,T^{-1}\epsilon^2+\beta_1\,T\epsilon^{-2\kappa}\,,$$
for all $T>0$ and for any $\epsilon>\norm{x-y}$. Since the right hand of this inequality is a continous function of $\epsilon>0$ we also have
$$\phi(x,y,T)\leq\alpha_1\,T^{-1}\norm{x-y}^2+\beta_1\,T\norm{x-y}^{-2\kappa}\,,$$
and the proof is achieved taking $T=\norm{x-y}^{1+\kappa}$.

\end{proof}\medbreak

\subsection{Homogeneity of the action potential.}

There is a property of homogeneity of
the action potential due to the homogeneity of the potential
function $U_\kappa$. We did not use this property in the
above proofs, but we think that it is useful to complete the
picture of the action potential. The proof can be done
reparametrizing conveniently homothetic paths of a given path.

\begin{proposition} If $\lambda>0$, then
$\phi(\lambda\,x,\lambda\,y)=\lambda^{1-\kappa}\phi(x,y)$ for all
$x,y\in E^N$.
\end{proposition}

\section{Weak KAM theory}\label{wkam}

The relationship between global solutions of the
Hamilton-Jacobi equation and globally minimizing solutions of the
corresponding Lagrangian flow is well known. Let us recall that the
\textsl{Hamiltonian}, defined on $T^*E^N=E^N\times (E^*)^N$ is the
function
$$H(x,p)=\frac{1}{2}\,\abs{p}^2-U_\kappa(x)\,,$$
where $\abs{p}$ denotes the dual norm of $p\in(E^*)^N$ with respect
to the norm on $E^N$ induced by the mass scalar product. More
precisely, if we identify the space $E$ with its dual
$E^*$ using the scalar product $<\;,\;>_E$, and $p=(p_1,\dots,p_N)\in (E^*)^N$, then
$$\abs{p}^2=\sum_{i=1}^N\,m_i^{-1}\,\norm{p_i}_E^2\,.$$  A
closely related function is the \textsl{total energy}, defined on
$TE^N$ as $\calE=H\circ\Leg$, where $\Leg:TE^N\to T^*E^N$ is the
Legendre transform $\Leg(\,x\,;\,v_1,\dots,v_N)=
(\,x\,;\,p_1,\dots,p_N)$, $p_i=m_iv_i$. It is easy to see that
$\calE$ is a first integral of the motion.

We will prove the existence of critical global (weak) solutions for
the Hamilton-Jacobi equation $H(x,d_xu)=c$. The critical value of
this Hamiltonian can be defined as the infimum of the values of
$c\in\R$ such that the Hamilton-Jacobi equation admits global
subsolutions. Since $\inf_{E^N} U_\kappa(x)=0$, and constants
functions are global subsolutions for $c=0$, it follows that the
critical value is $c=0$. Therefore, we are interested in global
solutions of
$$\abs{d_xu}^2=2\,U_\kappa(x)\,.\hspace{2cm}\textrm{(HJ)}$$
We will obtain global solutions as fixed points of a continuous
semigroup acting on the set of weak subsolutions, namely the
Lax-Oleinik semigroup. There are no new ideas in the method that we
apply here. In fact, we will follow the scheme introduced by Fathi
in \cite{Fathi-CRAS1}, with some adjustments to our setting. As we
have said in the introduction, the difference is that we consider a space of H\"{o}lder functions on which the semigroup acts, and theorem
\ref{potencial.holder} will assure that the method works with this
space.

\subsection{The Lax-Oleinik semigroup.}

Given a continuous function $u:E^N\to \R$ and $t>0$, we
define $T^-_tu:E^N\to [-\infty,+\infty)$ by
$$T^-_tu(x)=\inf\set{u(y)+\phi(x,y,t)\tq y\in E^N}\,.$$
We also define $T^-_0u=u$ for all function $u$. The semigroup property follows from the definition . In other words, for any function $u$ we have that $T^-_t(T^-_su)=T^-_{t+s}u$ for all $t,s\geq 0$. We will restrict the semigroup to the set $\calH$ of dominated functions. More precisely, we define
$$\calH=\set{u:E^N\to \R \tq u(x)-u(y)\leq
\phi(x,y)\textrm{ for all }x,y\in E^N}\,.$$ Notice that $u:E^N\to \R$ is in $\calH$ if and only if $u\leq T^-_tu$ for all $t\geq 0$. On the other hand, $u\leq v$ implies that $T^-_tu\leq T^-_tv$ for all $t\geq 0$. Therefore, the semigroup property implies that $T^-_tu\in \calH$ for all $u\in\calH$. Also notice that $\calH$ is convex, and nonempty since it contains all constant functions.

In the sequel, the set $\calH$ will be endowed the compact open
topology, that is to say, the topology generated by the sets
$$U_K(u,\epsilon)=\set{v\in\calH\tq \abs{v(x)-u(x)}<\epsilon\textrm{
for all }x\in K}\,,$$ with $u\in\calH$, $K\subset E^N$ compact,
and $\epsilon>0$.

\begin{proposition}\label{semigroup.cont}
The map $T^-:\calH\times [0,+\infty)\to\calH$, $(u,t)\mapsto
T^-_tu$ is continuous.
\end{proposition}

We will use the following lemma.

\begin{lemma}\label{phit.mayor.kd2t} For all $x,y\in E^N$ and $T>0$ we have
$\phi(x,y,T)\geq (m/2T)\norm{x-y}^2$, where
$m=\min\set{m_1,\dots,m_N}$.
\end{lemma}

\begin{proof} Let $r,s\in E$ and $\sigma:[0,T]\to E$ an absolutely
continuous curve such that $\sigma(0)=r$ and $\sigma(T)=s$. We
observe that
$$\norm{r-s}_E\leq \int_0^T\norm{\dot\sigma(t)}_E\,dt\leq
T^{1/2}\left(\int_0^T\norm{\dot\sigma(t)}_E^2\,dt\right)^{1/2}\,,$$
hence $$\norm{r-s}_E^2\leq
T\,\int_0^T\norm{\dot\sigma(t)}_E^2\,dt\,.$$ If $x=(r_1,\dots,r_N)$
and $y=(s_1,\dots,s_N)$ are two configurations, then we can choice
$i\in\set{1,\dots,N}$ such that $\norm{r_i-s_i}_E=\norm{x-y}$. Take
now $\gamma=(\gamma_1,\dots,\gamma_N)\in \calC(x,y,T)$. By the
previous observation we have,
$$A(\gamma)\geq
(m_i/2)\,\int_0^T\norm{\dot\gamma_i(t)}_E^2\,dt\geq
(m_i/2T)\,\norm{r_i-s_i}_E^2\geq (m/2T)\,\norm{x-y}^2\,,$$ which
proves the lemma since $\phi(x,y,T)=\inf\set{A(\gamma)\tq
\gamma\in\calC(x,y,T)}$. \end{proof}\medbreak

\begin{proof}[Proof of proposition \ref{semigroup.cont}.] As a first step,
we show that given $R>0$ and $t>0$, there is a constant $k(R,t)>0$
such that
$$T^-_tu(x)=\inf\set{u(y)+\phi(x,y,t)\tq \norm{y-x}\leq k(R,t)}$$
for all $u\in\calH$ and all $x\in E^N$ with $\norm{x}\leq R$. To
see this, fix $R>0$, $t>0$, $u\in\calH$ and $x\in E^N$ such that
$\norm{x}\leq R$. Suppose that $y\in E^N$ is such that
$\norm{y-x}>1$ and $u(y)+\phi(x,y,t)\leq u(x)+\phi(x,x,t)$. Then, by
lemma \ref{phit.mayor.kd2t} and theorem \ref{potencial.holder} we
have $$\frac{m}{2t}\,\norm{y-x}^2\leq
\eta\,\norm{y-x}^{1-\kappa}+\phi(x,x,t)\,.$$ Therefore, using that
$\norm{y-x}>1$ and theorem \ref{upper.bound} we deduce
$$m\,\norm{y-x}^2 \leq
2\eta\,t\,\norm{y-x}+2\alpha R^2+2\beta\,t^2R^{-2\kappa}\,,$$ hence
that $\norm{y-x}\leq k_0(R,t)$ where
$$k_0(R,t)=\eta\,t/m+\left(\eta^2t^2/m^2
+2\alpha\,R^2/m+2\beta\,t^2R^{-2\kappa}/m\right)^{1/2}\,.$$ Setting
$k(R,t)=\max\set{1,k_0(R,t)}$, it follows that $u(y)+\phi(x,y,t)>
u(x)+\phi(x,x,t)$ for all $y\in E^N$ such that $\norm{y-x}>k(R,t)$,
and we conclude that
\begin{eqnarray*}
T^-_tu(x)&=&\inf\set{u(y)+ \phi(x,y,t)\tq y\in E^N}\\&=&
\inf\set{u(y)+ \phi(x,y,t)\tq \norm{y-x}\leq k(R,t)}\,.
\end{eqnarray*}

Let now $u,v\in \calH$ and $t>0$. Let $K\subset E^N$ be a compact
subset, and $R>0$ such that $\norm{x}\leq R$ for all $x\in K$. If we
set
\[K_t=\bigcup_{x\in\, K}\set{y\in E^N\tq \norm{y-x}\leq k(R,t)}\,,\]
then for all $x\in K$ we have
\[T^-_tv(x)=\inf\set{v(y)+\phi(x,y,t)\tq y\in K_t}\,.\]
On the other hand, since
\[v(y)\leq u(y)+\sup\set{\abs{u(y)-v(y)}\tq y\in K_t}\]
for all $y\in K_t$, we
deduce that $T^-_tv(x)\leq \inf\set{u(y)+\phi(x,y,t)\tq y\in K_t}+
\sup\set{\abs{u(y)-v(y)}\tq y\in K_t}$. Thus we have proved that
$T^-_tv(x)-T^-_tu(x)\leq\sup\set{\abs{u(y)-v(y)}\tq y\in K_t}$ for
all $x\in K$. Hence we have that
$$\abs{T^-_tv(x)-T^-_tu(x)}\leq\sup\set{\abs{u(y)-v(y)}\tq y\in K_t}$$
for all $x\in K$. Since the subset $K_t\subset
E^N$ is compact, this implies the continuity of the map $T^-_t$ for each $t\geq 0$.
It remains to prove the continuity of $T^-_t$ with respect to $t$.
Since $(T^-_t)_{t\geq 0}$ is a semigroup it suffices to prove the continuity at $t=0$.
Given $u\in\calH$, we have by corollary \ref{bound.phixxt} that
$$0 \leq T^-_tu(x)-u(x)\leq \phi(x,x,t)\leq \mu\, t^{(1-\kappa)/(1+\kappa)}$$ for all $x\in E^N$.
Therefore, $T^-_tu$ converges uniformly to $u$ when $t\to 0$. As we have said,
using the semigroup property we can deduce that $T^-_tu$ converges uniformly to
$T^-_{t_0}u$ when $t\to t_0$ for all $t_0\geq 0$.
\end{proof}\medbreak

\subsection{Proof of theorem \ref{wkam.thm}.}

\begin{proof} Let $\widehat{\calH}$ be the quotient space of $\calH$ by the
subspace of constants functions. Thus, $\widehat{\calH}$ is
homeomorphic to $\calH_0=\set{u\in\calH\tq u(0)=0}$. By theorem
\ref{potencial.holder} we have that dominated functions are
uniformly equicontinuous. It follows that $\calH_0$ is compact by
Ascoli's theorem. Therefore, $\widehat{\calH}$ is a compact, convex,
and nonempty subset of $\widehat{C^0}(E^N,\R)$, the quotient of the
vector space $C^0(E^N,\R)$ by the subspace of constant functions.
Notice that $\widehat{C^0}(M,\R)$ is endowed with the quotient
topology of the compact open topology on $C^0(M,\R)$. In particular,
$\widehat{C^0}(M,\R)$ is a locally convex topological vector space.

Since $T^-_t(u+c)=T^-_tu+c$ for all $c\in\R$, it is clear that the
semigroup $T^-$ defines canonically a continuous semigroup
$\widehat{T}^-_t:\widehat{\calH}\to\widehat{\calH}$. If we apply the
Schauder-Tykhonov theorem, see \cite{Dug} pages 414--415, we
conclude that $\widehat{T}^-_t$ has a fixed point in
$\widehat{\calH}$. That is to say, there is a function $u\in\calH$
such that $T^-_tu=u+c(t)$ for some function $c:[0,+\infty)\to\R$.
The semigroup property and the continuity of $T^-$ imply that
$c(t)=c(1)t$. Since $u\in\calH$, we have that $u\leq T^-_tu$ for
all $t\geq0$, hence we must have $c(1)\geq 0$. We will prove that
$c(1)=0$. Notice that $T^-_tu=u+c(1)t$ implies $u(x)-u(y)\leq
\phi(x,y,t)-c(1)t$ for all $x,y\in E^N$. Hence, by theorem
\ref{upper.bound} we have that $$u(x)-u(y)\leq
\alpha\,\frac{R^2}{t}+\left(\frac{\beta}{R^{\,2\kappa}}-c(1)\right)\,t$$
whenever $x$ and $y$ are contained in a ball of $E$ of radius $R>0$.
Since this must be true for $R$ and $t$ arbitrary large, we conclude
that $c(1)=0$. Therefore $T^-_tu=u$ for all $t\geq 0$.

It remains to prove that there are fixed points of $T^-$ which are
invariant by the group of symmetries. This can be done as in
\cite{Maderna} as follows. We define the $\calH_{inv}$ as the set
of functions in $\calH$ which are invariant by symmetries. Thus
$\calH_{inv}$ is also convex, closed and nonempty since constant
functions are invariant. Moreover, $\calH_{inv}$ is stable by the
Lax-Oleinik semigroup. Therefore, the quotient of this set by the
subspace of constants functions is also compact, convex, nonempty
and stable by the induced semigroup $\widehat{T}^-$. With the same
arguments as above we obtain an invariant fixed point. \end{proof}\medbreak

\subsection{Viscosity solutions and subsolutions.}

It is well known that the notion of dominated function is related to
a notion of subsolution of the Hamilton-Jacobi equation, namely the
notion of viscosity subsolution. On the other hand, viscosity
solutions (see below) can be detected as fixed points, modulo
constants, of the Lax-Oleinik semigroup. An introduction to the
subject of viscosity solutions can be found for instance in the
books \cite{BarC-D}, \cite{Barles} or \cite{Evans}. However, our
setting presents some technical differences, essentially due to the
fact that the potential function is infinite in the set of
configurations with collisions. The following is a little adaptation
of some results in section 5 of \cite{Fathi-Maderna}.

Recall that $u:E^N\to\R$ is a \textsl{viscosity subsolution} at
$x_0\in E^N$ of (HJ), if for each $C^1$ function $\psi:E^N\to\R$
such that $x_0$ is a maximum of $u-\psi$ we have
$\abs{d_{x_0}\psi}^2\leq 2\,U_\kappa(x_0)$. Given $V\subset E^N$, we
say that $u$ is a viscosity subsolution in $V$ if it is viscosity
subsolution at each $x\in V$. We remark that any function is
trivially a viscosity subsolution in $\Omega^{\,c}$, where
$\Omega\subset E^N$ denotes the set of configurations without
collisions.

Analogously, a function $u:E^N\to\R$ is said to be a
\textsl{viscosity supersolution} at $x_0\in E^N$ of (HJ), if for
each $C^1$ function $\psi:E^N\to\R$ such that $x_0$ is a minimum of
$u-\psi$ we have $\abs{d_{x_0}\psi}^2\geq 2\,U_\kappa(x_0)$. If
$x_0\in\Omega^{\,c}$, then $u$ is a viscosity supersolution at $x_0$
if and only if there are no $C^1$ functions $\psi$ such that $x_0$
is a minimum of $u-\psi$. As for subsolutions, given $V\subset E^N$,
we say that $u$ is a viscosity supersolution in $V$ if it is
viscosity supersolution at each $x\in V$.

We say that a continuous function $u:E^N\to\R$ is a
\textsl{viscosity solution} of (HJ) in $V\subset E^N$ if it is both
a subsolution and a supersolution in $V$. It is not difficult to see
that a viscosity solution $u$ satisfies (HJ) at each point $x\in V$
where the derivative $d_xu$ exists. We will prove the following.

\begin{proposition}\label{viscosidad} (1) Any $u\in\calH$
is almost everywhere differentiable and a viscosity subsolution of
Hamilton-Jacobi in $E^N$.\\ (2) If $u\in\calH$ is a fixed point of
the Lax-Oleinik semigroup, then $u$ is a viscosity solution of
Hamilton-Jacobi.
\end{proposition}

\begin{proof} The fact that dominated functions are differentiable
almost everywhere follows from proposition \ref{phi.loc.lip},
the fact that collisions are contained in a finite number of
affine subspaces and the Rademacher's theorem.
In order to prove that they are viscosity
subsolutions, take $u\in\calH$ and $\psi:E^N\to\R$ of class $C^1$
such that $u-\psi$ admits a maximum at some $x_0\in E^N$. Let $v\in
E^N$. For all $t>0$ we have $$\psi(x_0)-\psi(x_0-tv)\leq
u(x_0)-u(x_0-tv)\leq
\frac{1}{2}\int_{-t}^0\abs{v}^2ds+\int_{-t}^0U_\kappa(x_0+sv)\,ds\,.$$
Dividing by $t$ and taking the limit for $t\to 0$ we obtain
$$d_{x_0}\psi(v)\leq \frac{1}{2}\,\abs{v}^2+U_\kappa(x_0)\,.$$ If we
define $p_1,\dots,p_N\in E$ by the condition
$d_{x_0}\psi(v)=\sum_{i=1}^N <p_i,v_i>_E$ for all
$v=(v_1,\dots,v_N)\in E^N$, then we can write
$$\abs{d_{x_0}\psi}^2=\sum_{i=1}^Nm_i^{-1}\norm{p_i}^2_E=d_{x_0}\psi(w)$$
where $w=(w_1,\dots,w_N)$ and $m_iw_i=p_i$ for all $i=1,\dots,N$.
Therefore, using the last inequality with $v=w$ we obtain
$\abs{d_{x_0}\psi}^2\leq 2\,U_\kappa(x_0)$.

It remains to prove (2).
Suppose that $u\in\calH$ is such that
$T^-_tu=u$ for all $t>0$. Let $\psi:E^N\to\R$ be a $C^1$ function
such that $u-\psi$ has a minimum at some $x_0\in\Omega$.

With the same arguments as in the proof of proposition
\ref{semigroup.cont}, we deduce that there is a constant $k>0$ such
that
\[T^-_1u(x_0)=\inf\set{u(y)+\phi(x_0,y,1)\tq \norm{y-x_0}\leq k}\,.\]
Therefore, using theorem \ref{upper.bound} and the lower
semi-continuity of the Lagrangian action we can choose $y_0\in E^N$
such that $\norm{y_0-x_0}\leq k$ and a curve
$\gamma\in\calC(x_0,y_0,1)$ such that
$$u(x_0)=T^-_1u(x_0)=u(y_0)+\frac{1}{2}\int_0^1\abs{\dot\gamma(t)}^2dt+
\int_0^1U_\kappa(\gamma(t))\,dt\,.$$ In particular, since $u$ is a
dominated function we must have
$$u(x_0)-u(\gamma(t))=\frac{1}{2}\int_0^t\abs{\dot\gamma(s)}^2ds+
\int_0^tU_\kappa(\gamma(s))\,ds$$ for all $t\in [0,1]$, which says
that $\gamma$ is a calibrated curve for $u$.
Using that $x_0$ is a minimum for $u-\psi$ we conclude that $$\psi(x_0)-\psi(\gamma(t))\geq
u(x_0)-u(\gamma(t))=\frac{1}{2}\int_0^t\abs{\dot\gamma(s)}^2ds+
\int_0^tU_\kappa(\gamma(s))\,ds\,.$$

At this point we can prove that $x_0$ is a configuration without collisions.
In other words, $u$ is a viscosity supersolution at each collision
configuration $x_0$ because there are no $C^1$ test functions $\psi$
such that $u-\psi$ has a minimum in $x_0$. To see this we proceed as follows:

Let $K>0$ be a Lipschitz constant for the restriction of $\psi$ to some neighborhood of $x_0$.
Using the fact that $U_\kappa>0$ in the last inequality we deduce that
$$K\;\abs{\gamma(t)-\gamma(0)}\geq \frac{1}{2}\int_0^t\abs{\dot\gamma(s)}^2ds$$
for any $t>0$ small enough. Applying the Cauchy-Schwarz inequality we also have
$$t\;\int_0^t\;\abs{\dot\gamma(s)}^2\,ds\;\geq\;
\left(\int_0^t\abs{\dot\gamma(s)}\,ds\right)^2
\;\geq \abs{\gamma(t)-\gamma(0)}^2$$
hence
$$\abs{\gamma(t)-\gamma(0)}\leq 2Kt\;.$$
Using again the Lipschitz constant for $\psi$ and the previous inequality,
but neglecting this time the kinetic term, we obtain
$$2t\,K^2 \geq \int_0^t U_\kappa(\gamma(s))\,ds\;.$$
Dividing by $t$ and taking the
limit for $t\to 0$ we deduce that $U_\kappa(x_0)\leq 2K^2$,
meaning that $x_0\in\Omega$.
Since $x_0$ is a configuration without collisions,
there is $\delta>0$ such that $\gamma([0,\delta])\subset\Omega$.
Since $\gamma$ is minimizing, hence a solution of the Euler-Lagrange flow in $[0,\delta]$ we know that $\gamma$ is differentiable at $t=0$.
Dividing by $t$ and taking the
limit for $t\to 0$ in the inequality
$$\psi(x_0)-\psi(\gamma(t))\geq \frac{1}{2}\int_0^t\abs{\dot\gamma(s)}^2ds+
\int_0^tU_\kappa(\gamma(s))\,ds$$
we obtain $$d_{x_0}\psi(v)\geq
\frac{1}{2}\,\abs{v}^2+U_\kappa(x_0)\,,$$ where $v=-\dot\gamma(0)$.
On the other hand, always we have $2\,p(v)\leq\abs{p}^2+\abs{v}^2$
for $p\in (E^*)^N$ and $v\in E^N$. Thus we conclude that $\abs{d_{x_0}\psi}^2\geq
2\,U_\kappa(x_0)$. We have proved that $u$ is a viscosity
supersolution at $x_0$.
\end{proof}\medbreak

\subsection{Lax-Oleinik and weak KAM solutions.}

Following the analogy with the weak KAM theory for Tonelli Lagrangians on compact manifolds, we show that the fixed points of the Lax-Oleinik semigroup are the weak KAM solutions defined by Fathi in \cite{Fathi-CRAS1}. More precisely, we show that the fixed points of Lax-Oleinik semigroup are characterized by the following property: given any configuration $x\in E^N$, always we have a calibrated curve $\gamma_x:(-\infty,0]\to E^N$ such that $\gamma_x(0)=x$.

Recall that if $u:E^N\to\R$ is a dominated function, then a curve $\gamma:I\to E^N$ is said to be calibrated when satisfies $$u(\gamma(b))-u(\gamma(a))=A(\gamma|_{[a,b]})$$ for all compact interval $[a,b]\subset I$. In particular, the calibrated curves of a dominated function are free time minimizers, meaning that $$A(\gamma|_{[a,b]})=\phi(\gamma(a),\gamma(b))$$ for all compact interval $[a,b]\subset I$.

Therefore, the fixed points of the Lax-Oleinik semigroup can be characterized in terms of calibrated curves as follows (recall that our Lagrangian is symmetric).

\begin{proposition}\label{fixpoints=calib}
 Let $u\in\calH$ be a dominated function. Then $u=T^-_tu$ for all $t>0$ if and only if, for each $x\in E^N$ there is a curve $\gamma_x:[0,+\infty)\to E^N$ with $\gamma_x(0)=x$ and such that $u(x)=u(\gamma_x(t))+A(\gamma|_{[0,t]})$ for all $t>0$.
\end{proposition}

\begin{proof} Suppose first that the condition is satisfied. Take $x\in E^N$ and the corresponding calibrated curve $\gamma_x:[0,+\infty)\to E^N$ with $\gamma_x(0)=x$. Since $u\in\calH$ already we known that $u\leq T^-_tu$ for all $t>0$. On the other hand, if we fix $t>0$ we have $T^-_tu(x)\leq u(\gamma_x(t))+A(\gamma|_{[0,t]})=u(x)$. Therefore $u$ is a fixed point.

Suppose now that $u\in\calH$ is a fixed point of $(T^-_t)$. Given a configuration $x\in E^N$ and $t>0$ we have $$u(x)=T^-_tu(x)=\inf\set{u(y)+\phi(x,y,t)\mid y\in E^N}\;.$$ Using lemma \ref{phit.mayor.kd2t} and theorem \ref{potencial.holder} (as in the proof of proposition \ref{semigroup.cont}) we deduce that there is a constant $k>0$ (depending on $x$ and $t$) such that $$u(x)=T^-_tu(x)=\inf\set{u(y)+\phi(x,y,t)\mid y\in E^N\textrm{ and }\norm{y-x}\leq k}\;.$$ Therefore, using theorem \ref{upper.bound} and the lower semi-continuity of the Lagrangian action we can choose $y(x,t)\in E^N$
such that $\norm{y(x,t)-x}\leq k$ and a curve
$\gamma_{x,t}\in\calC(x,y(x,t),t)$ such that
$$u(x)=T^-_tu(x)=u(y(x,t))+A(\gamma_{x,t}).$$ For each positive integer $n>0$ we define the curve $\gamma_n:[0,n]\to E^N$ as the curve $\gamma_{x,n}$. Observe that if $m>n$ then $\gamma_m\mid_{[0,n]}$ minimizes the action in $\calC(x,\gamma_m(n),n)$. Now we apply theorem \ref{upper.bound} and once again lemma \ref{phit.mayor.kd2t} and we deduce that for a fixed positive integer $n>0$, the sequence $(A(\gamma_m\mid_{[0,n]}))_{m>n}$ is bounded. It is not difficult to see (using the Cauchy-Schwartz inequality) that an absolutely continuous curve $\gamma:I\to E^N$ with finite Lagrangian action must satisfies $\abs{\gamma(t)-\gamma(s)}\leq 2A(\gamma)\,\abs{t-s}^{1/2}$ for all $t,s\in I$. Then we can apply Ascoli's theorem and deduce the existence of a convergent subsequence of $(\gamma_m\mid_{[0,n]})_{m>n}$. By a diagonal process we can extract an increasing sequence of indexes $m_k\in \N$ such that, for each positive integer $n>0$, the sequence $(\gamma_{m_k}\mid_{[0,n]})_{m_k>n}$ converges uniformly, when $k\to\infty$.
Observe now that by construction, each curve $\gamma_{m_k}\mid_{[0,n]})_{m_k>n}$ calibrates the function $u$. Therefore the semi-continuity of the action implies that the curve $\gamma_x:[0,+\infty)\to E^N$ defined by $\gamma_x(t)=\lim_{k\to\infty}\gamma_{m_k}(t)$ is also calibrated.
\end{proof}\medbreak

We remark that for the Newtonian potential ($\kappa=1/2$), Marchal's theorem implies -- except of course in the collinear case ($\dim E=1$) -- that the calibrated curves of weak KAM solutions are true motions for $t>0$ since they must be contained in $\Omega$, the set of configurations without collisions. The dynamics of the free time minimizers of the Newtonian N-body problem is described in \cite{DaLuz-Maderna}.

\section{The Kepler problem}

Unfortunately, the proof of the weak KAM theorem do not give any explicit solution.
Nevertheless, we can give explicit solutions for the Hamilton-Jacobi equation of
the Kepler problem.

We will find first isometry invariant solutions when we have two bodies of unit mass
in a line ($N=2$, $m_1=m_2=1$, $k=1$) and a Newtonian potential ($\kappa=1/2$).
An invariant solution $u:\R^2\to\R$ must satisfy $u(x+z,y+z)=u(x,y)$ and
$u(x,y)=u(-x,-y)$ for all $x,y,z\in\R$. Therefore the solution must be of the
form $u(x,y)=f(\abs{x-y})$, and the Hamilton-Jacobi equation reads
$$u_x^2+u_y^2=2\,\abs{x-y}^{-1}\;.$$ Replacing $u(x,y)$ by $f(\abs{x-y})$ and solving the differential equation in $f$ we conclude that the unique invariant global solutions (up to an additive constant) are the functions $$u_\pm(x,y)=\pm\;2\abs{x-y}^{1/2}\;.$$
In fact, the positive solution is the unique invariant fixed point of the \textsl{forward} Lax-Oleinik semi-group $$T^+_tu(x)=\inf\set{u(y)-\phi(x,y,t)\mid y\in E^N}$$ and therefore, the negative one is the unique invariant fixed point of the backward semigroup $T^-_t$. Of course, since the Lagrangian is symmetric in speed, we have that $u\in \calH$ is a backward solution if and only if $-u$ is a forward solution.

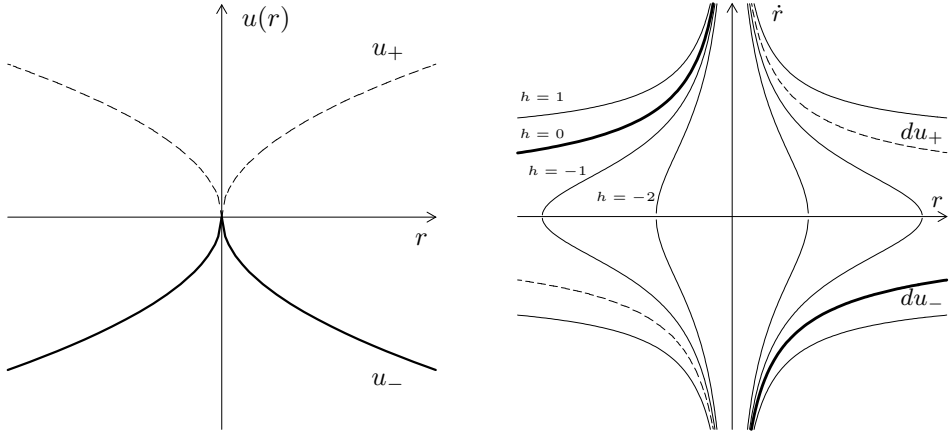
\begin{figure}[ht]
\input{umasumenos.tex}
\input{dumasdumenos.tex}
\caption{The two solutions and their derivatives for the 1-dimensional Kepler problem.}
\end{figure}

It is not difficult to see that we have also solutions invariant under translations,
in particular the function $b_+$ given by $b_+(x,y)=u_-(x,y)$ for $x\geq y$
and $b_+(x,y)=u_+(x,y)$ for $x\leq y$ is also a weak KAM solution.

For the planar Kepler problem it is convenient
to reduce first the problem by fixing
the center of mass at the origin, or equivalently,
to look for translation invariant solutions.
Since the configuration is then determined by the position of
the first body $x\in\R^2$, the problem reduces as usual to the center fix problem.
If we denote $x=(x_1,x_2)$ the position of the body,
then the Hamilton-Jacobi equation reads
$$u_{x_1}^2(x)+u_{x_2}^2(x)=2\norm{x}^{-1}\;.$$

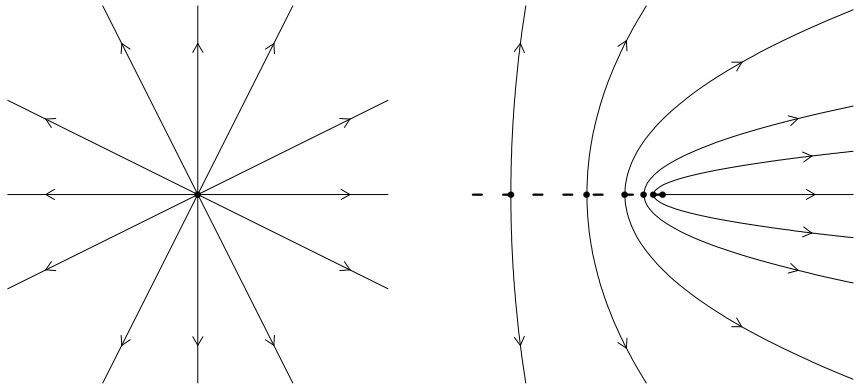
\begin{figure}[h]\label{calibKepler}
\input{SymKepler.tex}
\input{BusKepler.tex}
\caption{Calibrated curves of solutions for the planar Kepler problem.}
\end{figure}

Doubtlessly,
the simplest solution that we can give is  the rotation invariant solution $$u(x_1,x_2)=-(x_1^2+x_2^2)^{1/4}\,.$$ Its calibrated curves are all the parabolic homothetic motions, represented in the left side of figure \ref{calibKepler}. The half parabolas at the right side are the calibrated curves of a Buseman type solution, which is constant and not differentiable over the dashed line. A computation made by A. Venturelli shows that this last solution can be explicitly defined by the formula $$u(x_1,x_2)=-((x_1^2+x_2^2)^{1/2}+x_1)^{1/2}\,.$$

\textsl{Acknowledgements.} I would like to thank the hospitality of CIMAT, Guanajuato,
M\'{e}xico, where this research was done. I am specially grateful to Gonzalo Contreras and Renato Iturriaga for helpful discussions during the visit
of CIMAT. I also want to express my gratitude to Alain Chenciner and Andrea Venturelli,
for many discussions on the subject without which this paper would
not exist. Finally, I like to thank Marc Arcostanzo for pointing a
mistake in a draft version of this paper.

\end{document}

%% file: speedrepar.tex
\unitlength 1cm
\begin{picture}(11,6)(-0.5,-0.5)%
\color{black}\thinlines 
\path(0,1)(10,1)
\path(3,0)(3,5)
\path(9.8845,1.0667)(10,1)(9.8845,0.9333)
\path(2.9333,4.8845)(3,5)(3.0667,4.8845)
\path(2,1)(2,5)
\dashline{0.15}(0,2.5874)(0.102,2.6154)(0.2041,2.6454)(0.3061,2.6778)(0.4082,2.7129)(0.5102,2.7511)(0.6122,2.793)
(0.7143,2.8393)(0.8163,2.8907)(0.9184,2.9484)(1.0204,3.0138)(1.1224,3.089)(1.2245,3.1769)(1.2755,3.2268)
(1.3265,3.2817)(1.3776,3.3424)(1.4286,3.4101)(1.4796,3.4865)(1.5306,3.5735)(1.5816,3.6741)(1.6327,3.7926)
(1.6837,3.9353)(1.7092,4.0187)(1.7347,4.1125)(1.7602,4.2192)(1.7857,4.3422)(1.8112,4.4864)(1.8367,4.6593)
(1.8495,4.7599)(1.8622,4.8725)(1.875,5)
\dashline{0.15}(2.1252,5)(2.1301,4.947)(2.1429,4.8259)(2.1556,4.7183)(2.1684,4.622)(2.1939,4.4556)(2.2194,4.3161)
(2.2449,4.1967)(2.2704,4.0928)(2.2959,4.0013)(2.3214,3.9197)(2.3469,3.8463)(2.398,3.719)(2.449,3.6119)
(2.5,3.5198)(2.551,3.4395)(2.602,3.3686)(2.6531,3.3052)(2.7041,3.2481)(2.7551,3.1963)(2.8571,3.1055)(2.9592,3.028)
(3.0612,2.9608)(3.1633,2.9017)(3.2653,2.8491)(3.3673,2.8019)(3.4694,2.7592)(3.5714,2.7203)(3.6735,2.6846)
(3.7755,2.6517)(3.8776,2.6212)(3.9796,2.5928)(4.0816,2.5664)(4.1837,2.5416)(4.2857,2.5183)(4.3878,2.4964)
(4.4898,2.4756)(4.5918,2.456)(4.6939,2.4374)(4.7959,2.4197)(4.898,2.4028)(5,2.3867)(5.102,2.3713)(5.2041,2.3566)
(5.3061,2.3425)(5.4082,2.329)(5.5102,2.316)(5.6122,2.3035)(5.7143,2.2914)(5.8163,2.2798)(5.9184,2.2686)
(6.0204,2.2578)(6.1224,2.2473)(6.2245,2.2372)(6.3265,2.2274)(6.4286,2.2179)(6.5306,2.2087)(6.6327,2.1997)
(6.7347,2.1911)(6.8367,2.1826)(6.9388,2.1744)(7.0408,2.1664)(7.1429,2.1587)(7.2449,2.1511)(7.3469,2.1437)
(7.449,2.1366)(7.551,2.1296)(7.6531,2.1227)(7.7551,2.116)(7.8571,2.1095)(7.9592,2.1031)(8.0612,2.0969)
(8.1633,2.0908)(8.2653,2.0849)(8.3673,2.0791)(8.4694,2.0734)(8.5714,2.0678)(8.6735,2.0623)(8.7755,2.0569)
(8.8776,2.0517)(8.9796,2.0465)(9.0816,2.0415)(9.1837,2.0365)(9.2857,2.0317)(9.3878,2.0269)(9.4898,2.0222)
(9.5918,2.0176)(9.6939,2.0131)(9.7959,2.0087)(9.898,2.0043)(10,2)
\thicklines 
\path(0,2.5874)(0.102,2.6154)(0.2041,2.6454)(0.3061,2.6778)(0.4082,2.7129)(0.5102,2.7511)(0.6122,2.793)
(0.7143,2.8393)(0.8163,2.8907)(0.9184,2.9484)(1.0204,3.0138)(1.1224,3.089)(1.2245,3.1769)(1.2755,3.2268)
(1.3265,3.2817)(1.3776,3.3424)(1.4286,3.4101)(1.4796,3.4865)(1.5306,3.5735)(1.5816,3.6741)(1.6327,3.7926)
(1.6837,3.9353)(1.7092,4.0187)(1.7347,4.1125)(1.7602,4.2192)(1.7857,4.3422)(1.8112,4.4864)(1.8367,4.6593)
(1.8495,4.7599)(1.8622,4.8725)(1.875,5)
\path(2.1252,5)(2.1301,4.947)(2.1429,4.8259)(2.1556,4.7183)(2.1684,4.622)(2.1939,4.4556)(2.2194,4.3161)
(2.2449,4.1967)(2.2704,4.0928)(2.2959,4.0013)(2.3214,3.9197)(2.3469,3.8463)(2.398,3.719)(2.449,3.6119)
(2.5,3.5198)(2.551,3.4395)(2.602,3.3686)(2.6531,3.3052)(2.7041,3.2481)(2.7551,3.1963)(2.8571,3.1055)(2.9592,3.028)
(3.0612,2.9608)(3.1633,2.9017)(3.2653,2.8491)(3.3673,2.8019)(3.4694,2.7592)(3.5714,2.7758)(3.6735,2.8202)
(3.7755,2.8694)(3.8776,2.9245)(3.9796,2.9866)(4.0816,3.0576)(4.1837,3.14)(4.2857,3.2374)(4.3367,3.2933)
(4.3878,3.3553)(4.4388,3.4247)(4.4898,3.5029)(4.5408,3.5924)(4.5918,3.6962)(4.6429,3.8189)(4.6939,3.9676)
(4.7194,4.0549)(4.7449,4.1535)(4.7704,4.2662)(4.7959,4.397)(4.8214,4.5516)(4.8342,4.6404)(4.8469,4.7389)
(4.8597,4.8489)(4.8724,4.9732)(4.8749,5)
\path(5.1251,5)(5.1276,4.9732)(5.1403,4.8489)(5.1531,4.7389)(5.1658,4.6404)(5.1786,4.5516)(5.2041,4.397)
(5.2296,4.2662)(5.2551,4.1535)(5.2806,4.0549)(5.3061,3.9676)(5.3571,3.8189)(5.4082,3.6962)(5.4592,3.5924)
(5.5102,3.5029)(5.5612,3.4247)(5.6122,3.3553)(5.6633,3.2933)(5.7143,3.2374)(5.8163,3.14)(5.9184,3.0576)
(6.0204,3.0138)(6.1224,3.089)(6.2245,3.1769)(6.2755,3.2268)(6.3265,3.2817)(6.3776,3.3424)(6.4286,3.4101)
(6.4796,3.4865)(6.5306,3.5735)(6.5816,3.6741)(6.6327,3.7926)(6.6837,3.9353)(6.7092,4.0187)(6.7347,4.1125)
(6.7602,4.2192)(6.7857,4.3422)(6.8112,4.4864)(6.8367,4.6593)(6.8495,4.7599)(6.8622,4.8725)(6.875,5)
\path(7.1252,5)(7.1301,4.947)(7.1429,4.8259)(7.1556,4.7183)(7.1684,4.622)(7.1939,4.4556)(7.2194,4.3161)
(7.2449,4.1967)(7.2704,4.0928)(7.2959,4.0013)(7.3214,3.9197)(7.3469,3.8463)(7.398,3.719)(7.449,3.6119)
(7.5,3.5198)(7.551,3.4395)(7.602,3.3686)(7.6531,3.3052)(7.7041,3.2481)(7.7551,3.1963)(7.8571,3.1055)(7.9592,3.028)
(8.0612,2.9608)(8.1633,2.9017)(8.2653,2.8491)(8.3673,2.8019)(8.4694,2.7592)(8.5714,2.7203)(8.6735,2.6846)
(8.7755,2.6517)(8.8776,2.6212)(8.9796,2.5928)(9.0816,2.5664)(9.1837,2.5416)(9.2857,2.5183)(9.3878,2.4964)
(9.4898,2.4756)(9.5918,2.456)(9.6939,2.4374)(9.7959,2.4197)(9.898,2.4028)(10,2.3867)
\thinlines 
\dashline{0.15}(0,2.0455)(0.102,2.0506)(0.2041,2.0559)(0.3061,2.0612)(0.4082,2.0667)(0.5102,2.0722)(0.6122,2.0779)
(0.7143,2.0837)(0.8163,2.0896)(0.9184,2.0957)(1.0204,2.1019)(1.1224,2.1082)(1.2245,2.1147)(1.3265,2.1214)
(1.4286,2.1282)(1.5306,2.1351)(1.6327,2.1423)(1.7347,2.1496)(1.8367,2.1571)(1.9388,2.1649)(2.0408,2.1728)
(2.1429,2.181)(2.2449,2.1894)(2.3469,2.198)(2.449,2.2069)(2.551,2.216)(2.6531,2.2255)(2.7551,2.2352)(2.8571,2.2453)
(2.9592,2.2557)(3.0612,2.2664)(3.1633,2.2775)(3.2653,2.2891)(3.3673,2.301)(3.4694,2.3134)(3.5714,2.3264)
(3.6735,2.3398)(3.7755,2.3538)(3.8776,2.3684)(3.9796,2.3836)(4.0816,2.3995)(4.1837,2.4162)(4.2857,2.4338)
(4.3878,2.4522)(4.4898,2.4716)(4.5918,2.4921)(4.6939,2.5138)(4.7959,2.5368)(4.898,2.5613)(5,2.5874)(5.102,2.6154)
(5.2041,2.6454)(5.3061,2.6778)(5.4082,2.7129)(5.5102,2.7511)(5.6122,2.793)(5.7143,2.8393)(5.8163,2.8907)
(5.9184,2.9484)(6.0204,3.0138)(6.1224,3.089)(6.2245,3.1769)(6.2755,3.2268)(6.3265,3.2817)(6.3776,3.3424)
(6.4286,3.4101)(6.4796,3.4865)(6.5306,3.5735)(6.5816,3.6741)(6.6327,3.7926)(6.6837,3.9353)(6.7092,4.0187)
(6.7347,4.1125)(6.7602,4.2192)(6.7857,4.3422)(6.8112,4.4864)(6.8367,4.6593)(6.8495,4.7599)(6.8622,4.8725)
(6.875,5)
\dashline{0.15}(7.1252,5)(7.1301,4.947)(7.1429,4.8259)(7.1556,4.7183)(7.1684,4.622)(7.1939,4.4556)(7.2194,4.3161)
(7.2449,4.1967)(7.2704,4.0928)(7.2959,4.0013)(7.3214,3.9197)(7.3469,3.8463)(7.398,3.719)(7.449,3.6119)
(7.5,3.5198)(7.551,3.4395)(7.602,3.3686)(7.6531,3.3052)(7.7041,3.2481)(7.7551,3.1963)(7.8571,3.1055)(7.9592,3.028)
(8.0612,2.9608)(8.1633,2.9017)(8.2653,2.8491)(8.3673,2.8019)(8.4694,2.7592)(8.5714,2.7203)(8.6735,2.6846)
(8.7755,2.6517)(8.8776,2.6212)(8.9796,2.5928)(9.0816,2.5664)(9.1837,2.5416)(9.2857,2.5183)(9.3878,2.4964)
(9.4898,2.4756)(9.5918,2.456)(9.6939,2.4374)(9.7959,2.4197)(9.898,2.4028)(10,2.3867)
\dashline{0.15}(0,2.1696)(0.102,2.1777)(0.2041,2.186)(0.3061,2.1945)(0.4082,2.2033)(0.5102,2.2123)(0.6122,2.2217)
(0.7143,2.2313)(0.8163,2.2412)(0.9184,2.2515)(1.0204,2.2621)(1.1224,2.273)(1.2245,2.2844)(1.3265,2.2962)
(1.4286,2.3084)(1.5306,2.3211)(1.6327,2.3343)(1.7347,2.3481)(1.8367,2.3624)(1.9388,2.3774)(2.0408,2.3931)
(2.1429,2.4095)(2.2449,2.4266)(2.3469,2.4447)(2.449,2.4637)(2.551,2.4838)(2.6531,2.505)(2.7551,2.5274)
(2.8571,2.5513)(2.9592,2.5767)(3.0612,2.6039)(3.1633,2.6331)(3.2653,2.6645)(3.3673,2.6985)(3.4694,2.7354)
(3.5714,2.7758)(3.6735,2.8202)(3.7755,2.8694)(3.8776,2.9245)(3.9796,2.9866)(4.0816,3.0576)(4.1837,3.14)
(4.2857,3.2374)(4.3367,3.2933)(4.3878,3.3553)(4.4388,3.4247)(4.4898,3.5029)(4.5408,3.5924)(4.5918,3.6962)
(4.6429,3.8189)(4.6939,3.9676)(4.7194,4.0549)(4.7449,4.1535)(4.7704,4.2662)(4.7959,4.397)(4.8214,4.5516)
(4.8342,4.6404)(4.8469,4.7389)(4.8597,4.8489)(4.8724,4.9732)(4.8749,5)
\dashline{0.15}(5.1251,5)(5.1276,4.9732)(5.1403,4.8489)(5.1531,4.7389)(5.1658,4.6404)(5.1786,4.5516)(5.2041,4.397)
(5.2296,4.2662)(5.2551,4.1535)(5.2806,4.0549)(5.3061,3.9676)(5.3571,3.8189)(5.4082,3.6962)(5.4592,3.5924)
(5.5102,3.5029)(5.5612,3.4247)(5.6122,3.3553)(5.6633,3.2933)(5.7143,3.2374)(5.8163,3.14)(5.9184,3.0576)
(6.0204,2.9866)(6.1224,2.9245)(6.2245,2.8694)(6.3265,2.8202)(6.4286,2.7758)(6.5306,2.7354)(6.6327,2.6985)
(6.7347,2.6645)(6.8367,2.6331)(6.9388,2.6039)(7.0408,2.5767)(7.1429,2.5513)(7.2449,2.5274)(7.3469,2.505)
(7.449,2.4838)(7.551,2.4637)(7.6531,2.4447)(7.7551,2.4266)(7.8571,2.4095)(7.9592,2.3931)(8.0612,2.3774)
(8.1633,2.3624)(8.2653,2.3481)(8.3673,2.3343)(8.4694,2.3211)(8.5714,2.3084)(8.6735,2.2962)(8.7755,2.2844)
(8.8776,2.273)(8.9796,2.2621)(9.0816,2.2515)(9.1837,2.2412)(9.2857,2.2313)(9.3878,2.2217)(9.4898,2.2123)
(9.5918,2.2033)(9.6939,2.1945)(9.7959,2.186)(9.898,2.1777)(10,2.1696)
\path(5,1)(5,5)
\path(7,1)(7,5)
\put(2,1){\makebox(0,0){\tiny$\bullet$}}
\put(5,1){\makebox(0,0){\tiny$\bullet$}}
\put(7,1){\makebox(0,0){\tiny$\bullet$}}
\put(2,0.6){\makebox(0,0){\normalsize $b_1$}}
\put(7,0.6){\makebox(0,0){\normalsize $b_3$}}
\put(9.5,0.6){\makebox(0,0){\normalsize $t$}}
\put(5,0.6){\makebox(0,0){\normalsize $b_2$}}
\put(3.4,4.7){\makebox(0,0){\normalsize $f_{b,c}$}}
\end{picture}%

%% file: repar.tex
\unitlength 1cm
\begin{picture}(11,8)(-0.5,-0.5)%
\color{black}\thinlines 
\dashline{0.15}(2.0026,1.0187)(2.0153,1.0616)(2.0408,1.1185)(2.0918,1.2036)(2.1429,1.2733)(2.1939,1.335)
(2.2449,1.3914)(2.2959,1.4441)(2.3469,1.4937)(2.449,1.5863)(2.551,1.6721)(2.6531,1.7527)(2.7551,1.8292)
(2.8571,1.9023)(2.9592,1.9726)(3.0612,2.0404)(3.1633,2.1061)(3.2653,2.1699)(3.3673,2.2319)(3.4694,2.2925)
(3.5714,2.3516)(3.6735,2.4095)(3.7755,2.4663)(3.8776,2.5219)(3.9796,2.5766)(4.0816,2.6303)(4.1837,2.6832)
(4.2857,2.7352)(4.3878,2.7865)(4.4898,2.837)(4.5918,2.8869)(4.6939,2.9361)(4.7959,2.9846)(4.898,3.0326)
(5,3.0801)(5.102,3.127)(5.2041,3.1734)(5.3061,3.2193)(5.4082,3.2647)(5.5102,3.3097)(5.6122,3.3542)(5.7143,3.3984)
(5.8163,3.4421)(5.9184,3.4854)(6.0204,3.5284)(6.1224,3.571)(6.2245,3.6133)(6.3265,3.6552)(6.4286,3.6968)
(6.5306,3.738)(6.6327,3.779)(6.7347,3.8196)(6.8367,3.86)(6.9388,3.9001)(7.0408,3.9399)(7.1429,3.9795)
(7.2449,4.0187)(7.3469,4.0578)(7.449,4.0965)(7.551,4.1351)(7.6531,4.1734)(7.7551,4.2115)(7.8571,4.2493)
(7.9592,4.2869)(8.0612,4.3244)(8.1633,4.3616)(8.2653,4.3986)(8.3673,4.4354)(8.4694,4.472)(8.5714,4.5084)
(8.6735,4.5446)(8.7755,4.5806)(8.8776,4.6165)(8.9796,4.6522)(9.0816,4.6877)(9.1837,4.723)(9.2857,4.7582)
(9.3878,4.7932)(9.4898,4.8281)(9.5918,4.8628)(9.6939,4.8973)(9.7959,4.9317)(9.898,4.9659)(10,5)
\thicklines 
\path(0.9996,0)(1.0612,0.0412)(1.1633,0.112)(1.2653,0.1858)(1.3673,0.263)(1.4694,0.3446)(1.5714,0.4316)
(1.6735,0.5258)(1.7245,0.5766)(1.7755,0.6306)(1.8265,0.689)(1.8776,0.7534)(1.9286,0.8278)(1.9796,0.9253)
(1.9923,0.9612)(1.9987,0.9882)
\thinlines 
\dashline{0.15}(0,0.706)(0.0408,0.7219)(0.1429,0.7619)(0.2449,0.8023)(0.3469,0.8429)(0.449,0.8838)(0.551,0.925)
(0.6531,0.9665)(0.7551,1.0083)(0.8571,1.0505)(0.9592,1.093)(1.0612,1.1359)(1.1633,1.1792)(1.2653,1.2228)
(1.3673,1.2669)(1.4694,1.3114)(1.5714,1.3563)(1.6735,1.4016)(1.7755,1.4474)(1.8776,1.4937)(1.9796,1.5405)
(2.0816,1.5878)(2.1837,1.6357)(2.2857,1.6842)(2.3878,1.7333)(2.4898,1.783)(2.5918,1.8334)(2.6939,1.8845)
(2.7959,1.9364)(2.898,1.9891)(3,2.0426)(3.102,2.0971)(3.2041,2.1525)(3.3061,2.209)(3.4082,2.2667)(3.5102,2.3256)
(3.6122,2.3858)(3.7143,2.4476)(3.8163,2.511)(3.9184,2.5763)(4.0204,2.6437)(4.1224,2.7134)(4.2245,2.7859)
(4.3265,2.8617)(4.4286,2.9414)(4.5306,3.026)(4.6327,3.1171)(4.7347,3.2171)(4.7857,3.2719)(4.8367,3.3313)
(4.8878,3.3973)(4.9388,3.4747)(4.9898,3.583)(4.9962,3.6055)
\dashline{0.15}(5.0064,3.6644)(5.0128,3.6846)(5.0255,3.7167)(5.051,3.7676)(5.102,3.8484)(5.1531,3.9161)
(5.2041,3.9766)(5.2551,4.0322)(5.3061,4.0842)(5.4082,4.1802)(5.5102,4.2685)(5.6122,4.351)(5.7143,4.4291)
(5.8163,4.5035)(5.9184,4.5748)(6.0204,4.6436)(6.1224,4.7101)(6.2245,4.7746)(6.3265,4.8373)(6.4286,4.8984)
(6.5306,4.9581)(6.6327,5.0165)(6.7347,5.0737)(6.8367,5.1298)(6.9388,5.1848)(7.0408,5.2389)(7.1429,5.2921)
(7.2449,5.3445)(7.3469,5.396)(7.449,5.4469)(7.551,5.497)(7.6531,5.5465)(7.7551,5.5953)(7.8571,5.6435)
(7.9592,5.6912)(8.0612,5.7383)(8.1633,5.7849)(8.2653,5.831)(8.3673,5.8766)(8.4694,5.9218)(8.5714,5.9665)
(8.6735,6.0108)(8.7755,6.0547)(8.8776,6.0982)(8.9796,6.1413)(9.0816,6.184)(9.1837,6.2264)(9.2857,6.2684)
(9.3878,6.3102)(9.4898,6.3516)(9.5918,6.3926)(9.6939,6.4334)(9.7959,6.4739)(9.898,6.5141)(10,6.554)
\dashline{0.15}(0,1.9807)(0.102,2.0163)(0.2041,2.0522)(0.3061,2.0882)(0.4082,2.1244)(0.5102,2.1607)(0.6122,2.1973)
(0.7143,2.2341)(0.8163,2.271)(0.9184,2.3082)(1.0204,2.3456)(1.1224,2.3832)(1.2245,2.421)(1.3265,2.459)
(1.4286,2.4972)(1.5306,2.5357)(1.6327,2.5745)(1.7347,2.6134)(1.8367,2.6527)(1.9388,2.6922)(2.0408,2.7319)
(2.1429,2.7719)(2.2449,2.8123)(2.3469,2.8529)(2.449,2.8938)(2.551,2.935)(2.6531,2.9765)(2.7551,3.0183)
(2.8571,3.0605)(2.9592,3.103)(3.0612,3.1459)(3.1633,3.1892)(3.2653,3.2328)(3.3673,3.2769)(3.4694,3.3214)
(3.5714,3.3663)(3.6735,3.4116)(3.7755,3.4574)(3.8776,3.5037)(3.9796,3.5505)(4.0816,3.5978)(4.1837,3.6457)
(4.2857,3.6942)(4.3878,3.7433)(4.4898,3.793)(4.5918,3.8434)(4.6939,3.8945)(4.7959,3.9464)(4.898,3.9991)
(5,4.0526)(5.102,4.1071)(5.2041,4.1625)(5.3061,4.219)(5.4082,4.2767)(5.5102,4.3356)(5.6122,4.3958)(5.7143,4.4576)
(5.8163,4.521)(5.9184,4.5863)(6.0204,4.6537)(6.1224,4.7234)(6.2245,4.7959)(6.3265,4.8717)(6.4286,4.9514)
(6.5306,5.036)(6.6327,5.1271)(6.7347,5.2271)(6.7857,5.2819)(6.8367,5.3413)(6.8878,5.4073)(6.9388,5.4847)
(6.9898,5.593)(6.9962,5.6155)
\put(2,2){\makebox(0,0){\tiny$\bullet$}}
\put(2,1){\makebox(0,0){\tiny$\bullet$}}
\put(3,1){\makebox(0,0){\tiny$\bullet$}}
\put(5,2){\makebox(0,0){\tiny$\bullet$}}
\put(5,3.63){\makebox(0,0){\tiny$\bullet$}}
\put(3,3.63){\makebox(0,0){\tiny$\bullet$}}
\put(7,2){\makebox(0,0){\tiny$\bullet$}}
\put(7,5.64){\makebox(0,0){\tiny$\bullet$}}
\put(3,5.64){\makebox(0,0){\tiny$\bullet$}}
\thicklines 
\path(7.0026,5.6587)(7.0153,5.7016)(7.0408,5.7585)(7.0918,5.8436)(7.1429,5.9133)(7.1939,5.975)(7.2449,6.0314)
(7.2959,6.0841)(7.3469,6.1337)(7.449,6.2263)(7.551,6.3121)(7.6531,6.3927)(7.7551,6.4692)(7.8571,6.5423)
(7.9592,6.6126)(8.0612,6.6804)(8.1633,6.7461)(8.2653,6.8099)(8.3673,6.8719)(8.4694,6.9325)(8.5714,6.9916)
(8.5862,7)
\put(3.6,6.75){\makebox(0,0){\normalsize $F(t)$}}
\path(2.001,1.0097)(2.0019,1.0154)(2.0038,1.0245)(2.0077,1.0388)(2.0153,1.0616)(2.023,1.0808)(2.0306,1.0979)
(2.0383,1.1136)(2.0459,1.1282)(2.0536,1.1421)(2.0612,1.1553)(2.0689,1.168)(2.0765,1.1803)(2.0842,1.1921)
(2.0918,1.2036)(2.0995,1.2147)(2.1071,1.2256)(2.1148,1.2362)(2.1224,1.2466)(2.1301,1.2568)(2.1378,1.2667)
(2.1454,1.2765)(2.1531,1.2861)(2.1607,1.2956)(2.1684,1.3049)(2.176,1.3141)(2.1837,1.3231)(2.1913,1.332)
(2.199,1.3408)(2.2066,1.3495)(2.2143,1.3581)(2.2219,1.3666)(2.2296,1.3749)(2.2372,1.3832)(2.2449,1.3914)
(2.2526,1.3995)(2.2602,1.4076)(2.2679,1.4155)(2.2755,1.4234)(2.2832,1.4312)(2.2908,1.4389)(2.3061,1.4542)
(2.3214,1.4692)(2.3367,1.484)(2.352,1.4986)(2.3673,1.5129)(2.3827,1.5271)(2.398,1.541)(2.4133,1.5548)
(2.4286,1.5684)(2.4439,1.5819)(2.4592,1.5952)(2.4745,1.6083)(2.4898,1.6214)(2.5051,1.6342)(2.5204,1.647)
(2.5357,1.6596)(2.551,1.6721)(2.5663,1.6845)(2.5816,1.6968)(2.5969,1.709)(2.6122,1.721)(2.6276,1.733)
(2.6429,1.7449)(2.6582,1.7566)(2.6735,1.7683)(2.6888,1.7799)(2.7041,1.7914)(2.7194,1.8029)(2.7347,1.8142)
(2.75,1.8255)(2.7653,1.8367)(2.7806,1.8478)(2.7959,1.8588)(2.8112,1.8698)(2.8265,1.8807)(2.8418,1.8916)
(2.8571,1.9023)(2.8724,1.913)(2.8878,1.9237)(2.9031,1.9343)(2.9184,1.9448)(2.9337,1.9553)(2.949,1.9657)
(2.9643,1.976)(2.9796,1.9863)(2.9949,1.9966)(3.0102,2.0068)(3.0255,2.0169)(3.0408,2.027)(3.0561,2.0371)
(3.0714,2.0471)(3.0867,2.057)(3.102,2.0669)(3.1173,2.0768)(3.1327,2.0866)(3.148,2.0964)(3.1633,2.1061)
(3.1786,2.1158)(3.1939,2.1254)
\path(3.1939,2.1254)(3.2092,2.135)(3.2245,2.1446)(3.2398,2.1541)(3.2551,2.1636)(3.2704,2.173)(3.2857,2.1824)
(3.301,2.1918)(3.3163,2.2011)(3.3316,2.2104)(3.3469,2.2196)(3.3622,2.2289)(3.3776,2.2381)(3.3929,2.2472)
(3.4082,2.2563)(3.4235,2.2654)(3.4388,2.2745)(3.4541,2.2835)(3.4694,2.2925)(3.4847,2.3014)(3.5,2.3104)
\path(3.5,2.3196)(3.5153,2.3286)(3.5306,2.3375)(3.5459,2.3465)(3.5612,2.3555)(3.5765,2.3646)(3.5918,2.3737)
(3.6071,2.3828)(3.6224,2.3919)(3.6378,2.4011)(3.6531,2.4104)(3.6684,2.4196)(3.6837,2.4289)(3.699,2.4382)
(3.7143,2.4476)(3.7296,2.457)(3.7449,2.4664)(3.7602,2.4759)(3.7755,2.4854)(3.7908,2.495)(3.8061,2.5046)
(3.8214,2.5142)(3.8367,2.5239)(3.852,2.5336)(3.8673,2.5434)(3.8827,2.5532)(3.898,2.5631)(3.9133,2.573)
(3.9286,2.5829)(3.9439,2.5929)(3.9592,2.603)(3.9745,2.6131)(3.9898,2.6232)(4.0051,2.6334)(4.0204,2.6437)
(4.0357,2.654)(4.051,2.6643)(4.0663,2.6747)(4.0816,2.6852)(4.0969,2.6957)(4.1122,2.7063)(4.1276,2.717)
(4.1429,2.7277)(4.1582,2.7384)(4.1735,2.7493)(4.1888,2.7602)(4.2041,2.7712)(4.2194,2.7822)(4.2347,2.7933)
(4.25,2.8045)(4.2653,2.8158)(4.2806,2.8271)(4.2959,2.8386)(4.3112,2.8501)(4.3265,2.8617)(4.3418,2.8734)
(4.3571,2.8851)(4.3724,2.897)(4.3878,2.909)(4.4031,2.921)(4.4184,2.9332)(4.4337,2.9455)(4.449,2.9579)
(4.4643,2.9704)(4.4796,2.983)(4.4949,2.9958)(4.5102,3.0086)(4.5255,3.0217)(4.5408,3.0348)(4.5561,3.0481)
(4.5714,3.0616)(4.5867,3.0752)(4.602,3.089)(4.6173,3.1029)(4.6327,3.1171)(4.648,3.1314)(4.6633,3.146)
(4.6786,3.1608)(4.6939,3.1758)(4.7092,3.1911)(4.7168,3.1988)(4.7245,3.2066)(4.7321,3.2145)(4.7398,3.2224)
(4.7474,3.2305)(4.7551,3.2386)(4.7628,3.2468)(4.7704,3.2551)(4.7781,3.2634)(4.7857,3.2719)(4.7934,3.2805)
(4.801,3.2892)(4.8087,3.298)(4.8163,3.3069)(4.824,3.3159)(4.8316,3.3251)(4.8393,3.3344)(4.8469,3.3439)
(4.8546,3.3535)(4.8622,3.3633)
\path(4.8622,3.3633)(4.8699,3.3732)(4.8776,3.3834)(4.8852,3.3938)(4.8929,3.4044)(4.9005,3.4153)(4.9082,3.4264)
(4.9158,3.4379)(4.9235,3.4497)(4.9311,3.462)(4.9388,3.4747)(4.9464,3.4879)(4.9541,3.5018)(4.9617,3.5164)
(4.9694,3.5321)(4.977,3.5492)(4.9847,3.5684)(4.9923,3.5912)(4.9962,3.6055)(5,3.63)
\path(5.0006,3.6374)(5.0013,3.6418)(5.0026,3.6487)(5.0051,3.6596)(5.0077,3.6688)(5.0102,3.677)(5.0153,3.6916)
(5.0204,3.7047)(5.0255,3.7167)(5.0306,3.7279)(5.0357,3.7384)(5.0408,3.7485)(5.0459,3.7582)(5.051,3.7676)
(5.0561,3.7766)(5.0612,3.7853)(5.0663,3.7939)(5.0714,3.8022)(5.0765,3.8103)(5.0816,3.8182)(5.0867,3.8259)
(5.0918,3.8336)(5.0969,3.841)(5.102,3.8484)(5.1071,3.8556)(5.1122,3.8627)(5.1173,3.8697)(5.1224,3.8766)
(5.1276,3.8834)(5.1327,3.8901)(5.1378,3.8967)(5.1429,3.9033)(5.148,3.9097)(5.1531,3.9161)(5.1582,3.9225)
(5.1633,3.9287)(5.1684,3.9349)(5.1735,3.941)(5.1786,3.9471)(5.1837,3.9531)(5.1888,3.9591)(5.1939,3.965)
(5.199,3.9708)(5.2041,3.9766)(5.2092,3.9824)(5.2143,3.9881)(5.2194,3.9938)(5.2245,3.9994)(5.2296,4.0049)
(5.2347,4.0105)(5.2398,4.016)(5.2449,4.0214)(5.25,4.0269)(5.2551,4.0322)(5.2602,4.0376)(5.2653,4.0429)
(5.2704,4.0482)(5.2755,4.0534)(5.2806,4.0586)(5.2857,4.0638)(5.2908,4.0689)(5.2959,4.0741)(5.3061,4.0842)
(5.3163,4.0943)(5.3265,4.1042)(5.3367,4.114)(5.3469,4.1237)(5.3571,4.1334)(5.3673,4.1429)(5.3776,4.1524)
(5.3878,4.1617)(5.398,4.171)(5.4082,4.1802)(5.4184,4.1894)(5.4286,4.1984)(5.4388,4.2074)(5.449,4.2163)
(5.4592,4.2252)(5.4694,4.234)(5.4796,4.2427)(5.4898,4.2514)(5.5,4.26)(5.5102,4.2685)(5.5204,4.277)(5.5306,4.2854)
(5.5408,4.2938)(5.551,4.3021)(5.5612,4.3104)(5.5714,4.3186)(5.5816,4.3268)(5.5918,4.3349)(5.602,4.343)
(5.6122,4.351)(5.6224,4.359)(5.6327,4.367)(5.6429,4.3749)(5.6531,4.3827)(5.6633,4.3905)(5.6735,4.3983)
(5.6837,4.4061)
\path(5.6837,4.4061)(5.6939,4.4138)(5.7041,4.4214)(5.7143,4.4291)(5.7245,4.4367)(5.7347,4.4442)(5.7449,4.4517)
(5.7551,4.4592)(5.7653,4.4667)(5.7755,4.4741)(5.7857,4.4815)(5.7959,4.4888)(5.8061,4.4962)(5.8163,4.5035)
(5.8265,4.5107)(5.8367,4.518)(5.8469,4.5252)(5.8571,4.5323)(5.8673,4.5395)(5.8776,4.5466)(5.8878,4.5537)
(5.898,4.5608)(5.9082,4.5678)(5.9184,4.5748)(5.9286,4.5818)(5.9388,4.5888)(5.949,4.5957)(5.9592,4.6026)
(5.9694,4.6095)(5.9796,4.6163)(5.9898,4.6232)(6,4.63)
\path(6,4.64)(6.0102,4.6468)(6.0204,4.6537)(6.0306,4.6605)(6.0408,4.6674)(6.051,4.6743)(6.0612,4.6812)
(6.0714,4.6882)(6.0816,4.6952)(6.0918,4.7022)(6.102,4.7092)(6.1122,4.7163)(6.1224,4.7234)(6.1327,4.7305)
(6.1429,4.7377)(6.1531,4.7448)(6.1633,4.752)(6.1735,4.7593)(6.1837,4.7665)(6.1939,4.7738)(6.2041,4.7812)
(6.2143,4.7885)(6.2245,4.7959)(6.2347,4.8033)(6.2449,4.8108)(6.2551,4.8183)(6.2653,4.8258)(6.2755,4.8333)
(6.2857,4.8409)(6.2959,4.8486)(6.3061,4.8562)(6.3163,4.8639)(6.3265,4.8717)(6.3367,4.8795)(6.3469,4.8873)
(6.3571,4.8951)(6.3673,4.903)(6.3776,4.911)(6.3878,4.919)(6.398,4.927)(6.4082,4.9351)(6.4184,4.9432)(6.4286,4.9514)
(6.4388,4.9596)(6.449,4.9679)(6.4592,4.9762)(6.4694,4.9846)(6.4796,4.993)(6.4898,5.0015)(6.5,5.01)(6.5102,5.0186)
(6.5204,5.0273)(6.5306,5.036)(6.5408,5.0448)(6.551,5.0537)(6.5612,5.0626)(6.5714,5.0716)(6.5816,5.0806)
(6.5918,5.0898)(6.602,5.099)(6.6122,5.1083)(6.6224,5.1176)(6.6327,5.1271)(6.6429,5.1366)(6.6531,5.1463)
(6.6633,5.156)(6.6735,5.1658)(6.6837,5.1757)(6.6939,5.1858)(6.7041,5.1959)(6.7092,5.2011)(6.7143,5.2062)
(6.7194,5.2114)(6.7245,5.2166)(6.7296,5.2218)(6.7347,5.2271)(6.7398,5.2324)(6.7449,5.2378)(6.75,5.2431)
(6.7551,5.2486)(6.7602,5.254)(6.7653,5.2595)(6.7704,5.2651)(6.7755,5.2706)(6.7806,5.2762)(6.7857,5.2819)
(6.7908,5.2876)(6.7959,5.2934)(6.801,5.2992)(6.8061,5.305)(6.8112,5.3109)(6.8163,5.3169)(6.8214,5.3229)
(6.8265,5.329)(6.8316,5.3351)(6.8367,5.3413)(6.8418,5.3475)(6.8469,5.3539)(6.852,5.3603)(6.8571,5.3667)
\path(6.8571,5.3667)(6.8622,5.3733)(6.8673,5.3799)(6.8724,5.3866)(6.8776,5.3934)(6.8827,5.4003)(6.8878,5.4073)
(6.8929,5.4144)(6.898,5.4216)(6.9031,5.429)(6.9082,5.4364)(6.9133,5.4441)(6.9184,5.4518)(6.9235,5.4597)
(6.9286,5.4678)(6.9337,5.4761)(6.9388,5.4847)(6.9439,5.4934)(6.949,5.5024)(6.9541,5.5118)(6.9592,5.5215)
(6.9643,5.5316)(6.9694,5.5421)(6.9745,5.5533)(6.9796,5.5653)(6.9847,5.5784)(6.9898,5.593)(6.9923,5.6012)
(6.9949,5.6104)(6.9974,5.6213)(6.9987,5.6282)(7,5.64)
\thinlines 
\path(0,2)(10,2)
\path(3,0)(3,7)
\path(9.8845,2.0667)(10,2)(9.8845,1.9333)
\path(2.9333,6.8845)(3,7)(3.0667,6.8845)
\put(9.5,1.6){\makebox(0,0){\normalsize $t$}}
\path(2,2)(2,1)(3,1)
\path(5,2)(5,3.63)(3,3.63)
\path(7,2)(7,5.64)(3,5.64)
\put(3.4,1){\makebox(0,0){\normalsize $a_1$}}
\put(2.6,3.63){\makebox(0,0){\normalsize $a_2$}}
\put(2.6,5.64){\makebox(0,0){\normalsize $a_3$}}
\put(2,2.4){\makebox(0,0){\normalsize $b_1$}}
\put(5,1.6){\makebox(0,0){\normalsize $b_2$}}
\put(7,1.6){\makebox(0,0){\normalsize $b_3$}}
\end{picture}%

%% file: umasumenos.tex
\unitlength 1cm
\begin{picture}(6.1289,6.6289)(-0,-0.5)%
\color{black}\thinlines 
\path(0,2.8145)(5.6289,2.8145)
\path(2.8145,0)(2.8145,5.6289)
\path(5.5135,2.8811)(5.6289,2.8145)(5.5135,2.7478)
\path(2.7478,5.5135)(2.8145,5.6289)(2.8811,5.5135)
\thicklines 
\path(0,0.7876)(0.1105,0.8276)(0.2908,0.8949)(0.471,0.9648)(0.6513,1.0373)(0.8315,1.113)(1.0118,1.1922)
(1.1921,1.2754)(1.3723,1.3634)(1.5526,1.4571)(1.7329,1.5578)(1.9131,1.6673)(2.0934,1.7884)(2.2737,1.9259)
(2.4539,2.0889)(2.5441,2.1861)(2.6342,2.3014)(2.7243,2.4517)(2.7694,2.558)(2.8145,2.8145)(2.8595,2.558)
(2.9046,2.4517)(2.9947,2.3014)(3.0849,2.1861)(3.175,2.0889)(3.3553,1.9259)(3.5355,1.7884)(3.7158,1.6673)
(3.8961,1.5578)(4.0763,1.4571)(4.2566,1.3634)(4.4369,1.2754)(4.6171,1.1922)(4.7974,1.113)(4.9776,1.0373)
(5.1579,0.9648)(5.3382,0.8949)(5.5184,0.8276)(5.6289,0.7876)
\thinlines 
\dashline{0.15}(0,4.8413)(0.1105,4.8014)(0.2908,4.734)(0.471,4.6642)(0.6513,4.5916)(0.8315,4.516)(1.0118,4.4368)
(1.1921,4.3535)(1.3723,4.2655)(1.5526,4.1718)(1.7329,4.0711)(1.9131,3.9616)(2.0934,3.8405)(2.2737,3.703)
(2.4539,3.54)(2.5441,3.4428)(2.6342,3.3275)(2.7243,3.1772)(2.7694,3.071)(2.8145,2.8145)(2.8595,3.071)
(2.9046,3.1772)(2.9947,3.3275)(3.0849,3.4428)(3.175,3.54)(3.3553,3.703)(3.5355,3.8405)(3.7158,3.9616)
(3.8961,4.0711)(4.0763,4.1718)(4.2566,4.2655)(4.4369,4.3535)(4.6171,4.4368)(4.7974,4.516)(4.9776,4.5916)
(5.1579,4.6642)(5.3382,4.734)(5.5184,4.8014)(5.6289,4.8413)
\put(5.0045,0.6245){\makebox(0,0){\normalsize $u_-$}}
\put(5.0045,5.0045){\makebox(0,0){\normalsize $u_+$}}
\put(5.4425,2.5225){\makebox(0,0){\normalsize $r$}}
\put(3.3985,5.4425){\makebox(0,0){\normalsize $u(r)$}}
\end{picture}%

%% file: dumasdumenos.tex
\unitlength 1cm
\begin{picture}(6.1447,6.6447)(-0.5,-0.5)%
\color{black}\thinlines 
\path(0,2.8224)(5.6447,2.8224)
\path(2.8224,0)(2.8224,5.6447)
\path(5.5293,2.889)(5.6447,2.8224)(5.5293,2.7557)
\path(2.7557,5.5293)(2.8224,5.6447)(2.889,5.5293)
\path(3.1095,5.6447)(3.1102,5.6415)(3.114,5.6254)(3.1178,5.6097)(3.1216,5.5943)(3.1254,5.5792)(3.1292,5.5644)
(3.133,5.5498)(3.1368,5.5356)(3.1406,5.5216)(3.1443,5.5078)(3.1481,5.4944)(3.1519,5.4811)(3.1557,5.4681)
(3.1595,5.4553)(3.1633,5.4428)(3.1671,5.4305)(3.1709,5.4184)(3.1746,5.4064)(3.1784,5.3947)(3.1822,5.3832)
(3.186,5.3719)(3.1936,5.3498)(3.2012,5.3284)(3.2087,5.3076)(3.2163,5.2875)(3.2239,5.268)(3.2315,5.2491)
(3.239,5.2307)(3.2466,5.2128)(3.2542,5.1955)(3.2618,5.1786)(3.2693,5.1622)(3.2769,5.1462)(3.2845,5.1306)
(3.2921,5.1154)(3.2996,5.1006)(3.3072,5.0862)(3.3148,5.0722)(3.3224,5.0584)(3.3299,5.045)(3.3375,5.032)
(3.3451,5.0192)(3.3527,5.0067)(3.3602,4.9945)(3.3678,4.9826)(3.3754,4.971)(3.383,4.9596)(3.3906,4.9484)
(3.3981,4.9375)(3.4057,4.9268)(3.4133,4.9163)(3.4209,4.9061)(3.4284,4.896)(3.436,4.8862)(3.4436,4.8765)
(3.4512,4.8671)(3.4587,4.8578)(3.4663,4.8487)(3.4739,4.8397)(3.4815,4.831)(3.489,4.8224)(3.4966,4.8139)
(3.5042,4.8056)(3.5118,4.7975)(3.5193,4.7895)(3.5269,4.7816)(3.5345,4.7739)(3.5496,4.7589)(3.5648,4.7443)
(3.5799,4.7302)(3.5951,4.7166)(3.6102,4.7034)(3.6254,4.6907)(3.6406,4.6783)(3.6557,4.6663)(3.6709,4.6546)
(3.686,4.6433)(3.7012,4.6323)(3.7163,4.6216)(3.7315,4.6112)(3.7466,4.6011)(3.7618,4.5913)(3.7769,4.5817)
(3.7921,4.5724)(3.8072,4.5633)(3.8224,4.5544)(3.8375,4.5458)(3.8527,4.5374)(3.8678,4.5291)(3.883,4.5211)
(3.8981,4.5133)(3.9133,4.5056)(3.9284,4.4981)(3.9436,4.4908)(3.9587,4.4837)(3.9739,4.4767)(3.989,4.4699)
(4.0042,4.4632)(4.0193,4.4567)
\path(4.0193,4.4567)(4.0345,4.4503)(4.0496,4.444)(4.0648,4.4378)(4.0799,4.4318)(4.0951,4.4259)(4.1102,4.4202)
(4.1254,4.4145)(4.1406,4.409)(4.1557,4.4035)(4.1709,4.3982)(4.186,4.3929)(4.2012,4.3878)(4.2163,4.3827)
(4.2315,4.3778)(4.2466,4.3729)(4.2618,4.3682)(4.2769,4.3635)(4.2921,4.3589)(4.3072,4.3543)(4.3224,4.3499)
(4.3375,4.3455)(4.3527,4.3412)(4.3678,4.337)(4.383,4.3329)(4.3981,4.3288)(4.4133,4.3247)(4.4284,4.3208)
(4.4436,4.3169)(4.4587,4.3131)(4.4739,4.3093)(4.489,4.3056)(4.5042,4.302)(4.5193,4.2984)(4.5345,4.2948)
(4.5496,4.2913)(4.5648,4.2879)(4.5799,4.2845)(4.5951,4.2812)(4.6102,4.2779)(4.6254,4.2747)(4.6406,4.2715)
(4.6557,4.2684)(4.6709,4.2653)(4.686,4.2622)(4.7012,4.2592)(4.7163,4.2562)(4.7315,4.2533)(4.7466,4.2504)
(4.7618,4.2476)(4.7769,4.2448)(4.7921,4.242)(4.8072,4.2393)(4.8224,4.2366)(4.8375,4.2339)(4.8527,4.2313)
(4.8678,4.2287)(4.883,4.2261)(4.8981,4.2236)(4.9133,4.2211)(4.9284,4.2187)(4.9436,4.2162)(4.9587,4.2138)
(4.9739,4.2115)(4.989,4.2091)(5.0042,4.2068)(5.0193,4.2045)(5.0345,4.2023)(5.0496,4.2)(5.0648,4.1978)
(5.0799,4.1957)(5.0951,4.1935)(5.1102,4.1914)(5.1254,4.1893)(5.1406,4.1872)(5.1557,4.1851)(5.1709,4.1831)
(5.186,4.1811)(5.2012,4.1791)(5.2163,4.1772)(5.2315,4.1752)(5.2466,4.1733)(5.2618,4.1714)(5.2769,4.1695)
(5.2921,4.1677)(5.3072,4.1658)(5.3224,4.164)(5.3375,4.1622)(5.3527,4.1604)(5.3678,4.1587)(5.383,4.1569)
(5.3981,4.1552)(5.4133,4.1535)(5.4284,4.1518)(5.4436,4.1502)(5.4587,4.1485)(5.4739,4.1469)(5.489,4.1452)
(5.5042,4.1436)(5.5193,4.1421)
\path(5.5193,4.1421)(5.5345,4.1405)(5.5496,4.1389)(5.5648,4.1374)(5.5799,4.1359)(5.5951,4.1344)(5.6102,4.1329)
(5.6254,4.1314)(5.6406,4.1299)(5.6447,4.1295)
\put(0.3224,3.9224){\makebox(0,0){\tiny $h=0$}}
\dashline{0.15}(3.0745,5.6447)(3.0968,5.5221)(3.136,5.3478)(3.1752,5.2033)(3.2144,5.0811)(3.2928,4.8843)
(3.3712,4.7314)(3.4496,4.6081)(3.528,4.506)(3.6064,4.4196)(3.7632,4.2804)(3.92,4.1722)(4.0768,4.0851)
(4.2336,4.0128)(4.3904,3.9518)(4.5472,3.8992)(4.704,3.8534)(4.8608,3.8129)(5.0176,3.7769)(5.1744,3.7445)
(5.3312,3.7152)(5.488,3.6886)(5.6447,3.6642)
\path(0,1.5152)(0.006,1.5146)(0.0367,1.5117)(0.0673,1.5086)(0.0979,1.5055)(0.1285,1.5024)(0.1591,1.4991)
(0.1897,1.4958)(0.2203,1.4925)(0.2509,1.489)(0.2816,1.4855)(0.3122,1.4819)(0.3428,1.4783)(0.3734,1.4745)
(0.404,1.4707)(0.4346,1.4668)(0.4652,1.4628)(0.4958,1.4587)(0.5265,1.4545)(0.5571,1.4502)(0.5877,1.4458)
(0.6183,1.4413)(0.6489,1.4367)(0.6795,1.4319)(0.7101,1.4271)(0.7407,1.4221)(0.7713,1.417)(0.802,1.4117)
(0.8326,1.4063)(0.8632,1.4008)(0.8938,1.3951)(0.9244,1.3893)(0.955,1.3833)(0.9856,1.3771)(1.0162,1.3707)
(1.0469,1.3641)(1.0775,1.3574)(1.1081,1.3504)(1.1387,1.3432)(1.1693,1.3358)(1.1999,1.3281)(1.2305,1.3202)
(1.2611,1.3121)(1.2918,1.3036)(1.3224,1.2948)(1.353,1.2858)(1.3836,1.2764)(1.4142,1.2666)(1.4448,1.2565)
(1.4754,1.246)(1.506,1.2351)(1.5367,1.2238)(1.5673,1.2119)(1.5979,1.1996)(1.6285,1.1868)(1.6591,1.1733)
(1.6897,1.1593)(1.7203,1.1446)(1.7509,1.1292)(1.7816,1.1131)(1.8122,1.0962)(1.8428,1.0783)(1.8734,1.0596)
(1.904,1.0397)(1.9346,1.0188)(1.9652,0.9966)(1.9958,0.9731)(2.0265,0.9481)(2.0571,0.9215)(2.0877,0.8931)
(2.1183,0.8626)(2.1336,0.8466)(2.1489,0.83)(2.1642,0.8127)(2.1795,0.7948)(2.1948,0.7762)(2.2101,0.7568)
(2.2254,0.7366)(2.2407,0.7156)(2.256,0.6936)(2.2713,0.6707)(2.2867,0.6467)(2.302,0.6217)(2.3173,0.5954)
(2.3326,0.5677)(2.3479,0.5387)(2.3632,0.5082)(2.3785,0.4759)(2.3938,0.4419)(2.4091,0.4059)(2.4244,0.3677)
(2.4397,0.327)(2.455,0.2838)(2.4627,0.2611)(2.4703,0.2376)(2.478,0.2133)(2.4856,0.1881)(2.4933,0.162)
(2.5009,0.135)(2.5086,0.1069)
\path(2.5086,0.1069)(2.5162,0.0777)(2.5239,0.0473)(2.5316,0.0157)(2.5352,0)
\path(3.0506,5.6447)(3.0543,5.6195)(3.0648,5.5517)(3.0754,5.488)(3.0859,5.428)(3.0964,5.3713)(3.107,5.3178)
(3.1175,5.267)(3.1281,5.2187)(3.1386,5.1728)(3.1492,5.1289)(3.1597,5.0871)(3.1702,5.0471)(3.1808,5.0088)
(3.1913,4.972)(3.2019,4.9367)(3.2124,4.9027)(3.2229,4.87)(3.244,4.8081)(3.2651,4.7504)(3.2862,4.6964)
(3.3073,4.6457)(3.3284,4.598)(3.3494,4.5529)(3.3705,4.5102)(3.3916,4.4696)(3.4127,4.4311)(3.4338,4.3944)
(3.4548,4.3593)(3.4759,4.3257)(3.497,4.2936)(3.5181,4.2627)(3.5392,4.2331)(3.5603,4.2045)(3.5813,4.177)
(3.6024,4.1505)(3.6235,4.1248)(3.6446,4.1)(3.6657,4.076)(3.6868,4.0527)(3.7078,4.0301)(3.7289,4.0082)
(3.75,3.9869)(3.7922,3.9459)(3.8343,3.907)(3.8765,3.8699)(3.9187,3.8345)(3.9608,3.8005)(4.003,3.7679)
(4.0452,3.7365)(4.0873,3.7062)(4.1295,3.6768)(4.1717,3.6484)(4.2138,3.6207)(4.256,3.5938)(4.2982,3.5675)
(4.3403,3.5418)(4.3825,3.5166)(4.4247,3.4919)(4.4668,3.4675)(4.509,3.4435)(4.5511,3.4198)(4.5933,3.3963)
(4.6355,3.3729)(4.6776,3.3496)(4.7198,3.3264)(4.762,3.3031)(4.8041,3.2798)(4.8463,3.2562)(4.8885,3.2323)
(4.9306,3.2079)(4.9728,3.183)(5.015,3.1573)(5.0571,3.1305)(5.0993,3.1023)(5.1415,3.0722)(5.1836,3.0392)
(5.2258,3.0017)(5.2469,2.9802)(5.268,2.9558)(5.289,2.9263)(5.3101,2.8851)(5.3207,2.8457)
\path(3.0231,5.6447)(3.024,5.6367)(3.0276,5.6056)(3.0312,5.5753)(3.0348,5.5456)(3.0384,5.5167)(3.042,5.4884)
(3.0456,5.4607)(3.0492,5.4336)(3.0528,5.407)(3.0564,5.3811)(3.06,5.3556)(3.0672,5.3063)(3.0744,5.2589)
(3.0816,5.2132)(3.0888,5.1692)(3.096,5.1267)(3.1032,5.0857)(3.1104,5.046)(3.1176,5.0076)(3.1248,4.9703)
(3.132,4.9342)(3.1392,4.8992)(3.1464,4.8651)(3.1536,4.832)(3.1608,4.7998)(3.168,4.7684)(3.1752,4.7378)
(3.1824,4.708)(3.1896,4.6789)(3.1968,4.6505)(3.204,4.6227)(3.2112,4.5955)(3.2184,4.5689)(3.2256,4.5429)
(3.2328,4.5174)(3.24,4.4925)(3.2472,4.468)(3.2544,4.444)(3.2616,4.4204)(3.2688,4.3973)(3.276,4.3745)(3.2832,4.3522)
(3.2904,4.3302)(3.2976,4.3086)(3.312,4.2663)(3.3264,4.2253)(3.3408,4.1855)(3.3552,4.1467)(3.3696,4.1088)
(3.384,4.0719)(3.3984,4.0357)(3.4128,4.0003)(3.4272,3.9656)(3.4416,3.9314)(3.456,3.8978)(3.4704,3.8647)
(3.4848,3.832)(3.4992,3.7997)(3.5136,3.7676)(3.528,3.7359)(3.5424,3.7043)(3.5568,3.6728)(3.5712,3.6415)
(3.5856,3.6101)(3.6,3.5787)(3.6144,3.5471)(3.6288,3.5153)(3.6432,3.4832)(3.6576,3.4506)(3.672,3.4174)
(3.6864,3.3835)(3.7008,3.3486)(3.7152,3.3124)(3.7296,3.2747)(3.744,3.2348)(3.7584,3.1922)(3.7656,3.1694)
(3.7728,3.1454)(3.78,3.12)(3.7872,3.0925)(3.7944,3.0624)(3.8016,3.0285)(3.8088,2.9884)(3.8124,2.9645)
(3.816,2.9359)(3.8196,2.8973)(3.8214,2.8671)
\path(3.1095,0)(3.1104,0.0038)(3.1176,0.034)(3.1248,0.0631)(3.132,0.0911)(3.1392,0.1181)(3.1464,0.1442)
(3.1536,0.1693)(3.1608,0.1937)(3.168,0.2172)(3.1752,0.2399)(3.1824,0.262)(3.1968,0.3041)(3.2112,0.3437)
(3.2256,0.381)(3.24,0.4163)(3.2544,0.4497)(3.2688,0.4814)(3.2832,0.5115)(3.2976,0.5401)(3.312,0.5674)
(3.3264,0.5934)(3.3408,0.6183)(3.3552,0.6421)(3.3696,0.6648)(3.384,0.6867)(3.3984,0.7076)(3.4128,0.7277)
(3.4272,0.7471)(3.4416,0.7657)(3.456,0.7836)(3.4704,0.8009)(3.4848,0.8175)(3.4992,0.8336)(3.5136,0.8492)
(3.528,0.8642)(3.5424,0.8787)(3.5712,0.9064)(3.6,0.9324)(3.6288,0.9569)(3.6576,0.9799)(3.6864,1.0017)
(3.7152,1.0223)(3.744,1.0419)(3.7728,1.0604)(3.8016,1.0781)(3.8304,1.0949)(3.8592,1.1109)(3.888,1.1262)
(3.9168,1.1409)(3.9456,1.1548)(3.9744,1.1683)(4.0032,1.1811)(4.032,1.1934)(4.0608,1.2053)(4.0896,1.2167)
(4.1184,1.2276)(4.1472,1.2382)(4.176,1.2483)(4.2048,1.2582)(4.2336,1.2676)(4.2624,1.2768)(4.2912,1.2856)
(4.32,1.2941)(4.3488,1.3024)(4.3776,1.3104)(4.4064,1.3182)(4.4352,1.3257)(4.464,1.333)(4.4928,1.34)(4.5216,1.3469)
(4.5504,1.3536)(4.5792,1.36)(4.608,1.3663)(4.6368,1.3724)(4.6656,1.3784)(4.6944,1.3842)(4.7232,1.3898)
(4.752,1.3953)(4.7808,1.4007)(4.8096,1.4059)(4.8384,1.411)(4.8672,1.4159)(4.896,1.4208)(4.9248,1.4255)
(4.9536,1.4301)(4.9824,1.4346)(5.0112,1.439)(5.04,1.4433)(5.0688,1.4475)(5.0976,1.4516)(5.1264,1.4556)
(5.1552,1.4595)(5.184,1.4634)(5.2128,1.4671)(5.2416,1.4708)(5.2704,1.4744)(5.2992,1.4779)(5.328,1.4814)
(5.3568,1.4848)
\path(5.3568,1.4848)(5.3856,1.4881)(5.4144,1.4914)(5.4432,1.4945)(5.472,1.4977)(5.5008,1.5007)(5.5296,1.5037)
(5.5584,1.5067)(5.5872,1.5096)(5.616,1.5124)(5.6447,1.5152)
\path(3.0506,0)(3.0528,0.0151)(3.06,0.0624)(3.0672,0.1076)(3.0744,0.1509)(3.0816,0.1925)(3.0888,0.2325)
(3.096,0.2709)(3.1032,0.3079)(3.1104,0.3436)(3.1176,0.378)(3.1248,0.4112)(3.132,0.4433)(3.1392,0.4744)
(3.1464,0.5044)(3.1536,0.5336)(3.1608,0.5618)(3.168,0.5892)(3.1752,0.6158)(3.1824,0.6416)(3.1896,0.6668)
(3.1968,0.6912)(3.204,0.715)(3.2112,0.7381)(3.2184,0.7607)(3.2256,0.7827)(3.24,0.8251)(3.2544,0.8654)
(3.2688,0.904)(3.2832,0.9408)(3.2976,0.9761)(3.312,1.0099)(3.3264,1.0424)(3.3408,1.0736)(3.3552,1.1037)
(3.3696,1.1327)(3.384,1.1607)(3.3984,1.1877)(3.4128,1.2138)(3.4272,1.2391)(3.4416,1.2635)(3.456,1.2873)
(3.4704,1.3103)(3.4848,1.3326)(3.4992,1.3543)(3.5136,1.3755)(3.528,1.396)(3.5424,1.416)(3.5568,1.4355)
(3.5712,1.4545)(3.5856,1.4731)(3.6,1.4912)(3.6144,1.5089)(3.6288,1.5262)(3.6432,1.5431)(3.6576,1.5596)
(3.672,1.5758)(3.6864,1.5916)(3.7008,1.6071)(3.7152,1.6223)(3.7296,1.6372)(3.744,1.6518)(3.7728,1.6803)
(3.8016,1.7077)(3.8304,1.7342)(3.8592,1.7598)(3.888,1.7846)(3.9168,1.8087)(3.9456,1.8321)(3.9744,1.8548)
(4.0032,1.877)(4.032,1.8986)(4.0608,1.9196)(4.0896,1.9402)(4.1184,1.9603)(4.1472,1.98)(4.176,1.9992)(4.2048,2.0182)
(4.2336,2.0367)(4.2624,2.055)(4.2912,2.0729)(4.32,2.0906)(4.3488,2.108)(4.3776,2.1252)(4.4064,2.1422)
(4.4352,2.159)(4.464,2.1756)(4.4928,2.192)(4.5216,2.2083)(4.5504,2.2245)(4.5792,2.2406)(4.608,2.2566)
(4.6368,2.2726)(4.6656,2.2885)(4.6944,2.3043)(4.7232,2.3202)(4.752,2.3361)(4.7808,2.352)(4.8096,2.368)
(4.8384,2.3841)
\path(4.8384,2.3841)(4.8672,2.4004)(4.896,2.4168)(4.9248,2.4334)(4.9536,2.4503)(4.9824,2.4675)(5.0112,2.4851)
(5.04,2.5032)(5.0688,2.5218)(5.0976,2.5412)(5.1264,2.5615)(5.1552,2.5829)(5.184,2.6058)(5.2128,2.6308)
(5.2416,2.6589)(5.256,2.6746)(5.2704,2.692)(5.2848,2.7118)(5.2992,2.7358)(5.3136,2.7692)(5.3208,2.7997)
\dashline{0.15}(0,1.9806)(0.1568,1.9562)(0.3136,1.9295)(0.4704,1.9002)(0.6272,1.8679)(0.784,1.8318)(0.9408,1.7914)
(1.0976,1.7455)(1.2544,1.693)(1.4112,1.6319)(1.568,1.5597)(1.7248,1.4725)(1.8816,1.3643)(2.0384,1.2252)
(2.1168,1.1388)(2.1952,1.0367)(2.2736,0.9134)(2.352,0.7604)(2.4304,0.5636)(2.4696,0.4414)(2.5088,0.297)
(2.548,0.1226)(2.5702,0)
\Thicklines 
\path(3.0735,0)(3.0749,0.0083)(3.0826,0.05)(3.0902,0.0898)(3.0979,0.1281)(3.1055,0.1647)(3.1132,0.1999)
(3.1208,0.2338)(3.1285,0.2663)(3.1361,0.2977)(3.1438,0.3279)(3.1515,0.3571)(3.1591,0.3853)(3.1668,0.4125)
(3.1744,0.4389)(3.1821,0.4643)(3.1897,0.489)(3.1974,0.513)(3.205,0.5362)(3.2203,0.5806)(3.2356,0.6225)
(3.2509,0.6621)(3.2662,0.6997)(3.2816,0.7354)(3.2969,0.7693)(3.3122,0.8016)(3.3275,0.8325)(3.3428,0.862)
(3.3581,0.8902)(3.3734,0.9172)(3.3887,0.9431)(3.404,0.968)(3.4193,0.992)(3.4346,1.015)(3.4499,1.0372)
(3.4652,1.0585)(3.4805,1.0792)(3.4958,1.0991)(3.5111,1.1183)(3.5265,1.137)(3.5418,1.155)(3.5571,1.1725)
(3.5724,1.1894)(3.5877,1.2058)(3.603,1.2217)(3.6183,1.2372)(3.6489,1.2668)(3.6795,1.2948)(3.7101,1.3214)
(3.7407,1.3466)(3.7713,1.3706)(3.802,1.3935)(3.8326,1.4153)(3.8632,1.4362)(3.8938,1.4561)(3.9244,1.4752)
(3.955,1.4935)(3.9856,1.5111)(4.0162,1.5281)(4.0469,1.5444)(4.0775,1.56)(4.1081,1.5752)(4.1387,1.5897)
(4.1693,1.6038)(4.1999,1.6174)(4.2305,1.6306)(4.2611,1.6434)(4.2918,1.6557)(4.3224,1.6677)(4.353,1.6793)
(4.3836,1.6905)(4.4142,1.7015)(4.4448,1.7121)(4.4754,1.7224)(4.506,1.7325)(4.5367,1.7422)(4.5673,1.7518)
(4.5979,1.761)(4.6285,1.7701)(4.6591,1.7789)(4.6897,1.7875)(4.7203,1.7958)(4.7509,1.804)(4.7816,1.812)
(4.8122,1.8198)(4.8428,1.8274)(4.8734,1.8349)(4.904,1.8422)(4.9346,1.8493)(4.9652,1.8563)(4.9958,1.8631)
(5.0265,1.8698)(5.0571,1.8763)(5.0877,1.8828)(5.1183,1.889)(5.1489,1.8952)(5.1795,1.9012)(5.2101,1.9072)
(5.2407,1.913)(5.2713,1.9187)
\path(5.2713,1.9187)(5.302,1.9243)(5.3326,1.9298)(5.3632,1.9352)(5.3938,1.9405)(5.4244,1.9457)(5.455,1.9508)
(5.4856,1.9558)(5.5162,1.9607)(5.5469,1.9656)(5.5775,1.9704)(5.6081,1.9751)(5.6387,1.9797)(5.6447,1.9806)
\thinlines 
\path(3.0231,0)(3.024,0.008)(3.0276,0.0391)(3.0312,0.0695)(3.0348,0.0991)(3.0384,0.1281)(3.042,0.1564)
(3.0456,0.1841)(3.0492,0.2112)(3.0528,0.2377)(3.0564,0.2637)(3.06,0.2891)(3.0672,0.3384)(3.0744,0.3859)
(3.0816,0.4315)(3.0888,0.4756)(3.096,0.518)(3.1032,0.5591)(3.1104,0.5988)(3.1176,0.6372)(3.1248,0.6744)
(3.132,0.7105)(3.1392,0.7456)(3.1464,0.7796)(3.1536,0.8127)(3.1608,0.845)(3.168,0.8763)(3.1752,0.9069)
(3.1824,0.9368)(3.1896,0.9659)(3.1968,0.9943)(3.204,1.0221)(3.2112,1.0492)(3.2184,1.0758)(3.2256,1.1018)
(3.2328,1.1273)(3.24,1.1523)(3.2472,1.1767)(3.2544,1.2008)(3.2616,1.2243)(3.2688,1.2475)(3.276,1.2702)
(3.2832,1.2926)(3.2904,1.3146)(3.2976,1.3362)(3.312,1.3784)(3.3264,1.4194)(3.3408,1.4593)(3.3552,1.4981)
(3.3696,1.5359)(3.384,1.5729)(3.3984,1.609)(3.4128,1.6444)(3.4272,1.6792)(3.4416,1.7133)(3.456,1.7469)
(3.4704,1.7801)(3.4848,1.8128)(3.4992,1.8451)(3.5136,1.8771)(3.528,1.9089)(3.5424,1.9405)(3.5568,1.9719)
(3.5712,2.0033)(3.5856,2.0346)(3.6,2.0661)(3.6144,2.0976)(3.6288,2.1294)(3.6432,2.1616)(3.6576,2.1942)
(3.672,2.2273)(3.6864,2.2613)(3.7008,2.2962)(3.7152,2.3323)(3.7296,2.3701)(3.744,2.4099)(3.7584,2.4526)
(3.7656,2.4753)(3.7728,2.4993)(3.78,2.5248)(3.7872,2.5522)(3.7944,2.5823)(3.8016,2.6163)(3.8088,2.6563)
(3.8124,2.6802)(3.816,2.7089)(3.8196,2.7474)(3.8214,2.7776)
\put(5.5224,3.0224){\makebox(0,0){\small $r$}}
\Thicklines 
\path(0,3.6642)(0.0576,3.6729)(0.1152,3.6819)(0.1728,3.6912)(0.2304,3.7008)(0.288,3.7107)(0.3456,3.721)
(0.4032,3.7316)(0.4608,3.7426)(0.5184,3.7541)(0.576,3.7659)(0.6336,3.7783)(0.6912,3.7911)(0.7488,3.8045)
(0.8064,3.8184)(0.864,3.8329)(0.9216,3.8481)(0.9792,3.864)(1.0368,3.8807)(1.0944,3.8982)(1.152,3.9166)
(1.2096,3.936)(1.2672,3.9564)(1.3248,3.978)(1.3824,4.0009)(1.44,4.0252)(1.4976,4.0511)(1.5552,4.0787)
(1.6128,4.1082)(1.6704,4.14)(1.728,4.1742)(1.7856,4.2113)(1.8432,4.2515)(1.9008,4.2955)(1.9584,4.3438)
(2.016,4.3972)(2.0448,4.4261)(2.0736,4.4567)(2.1024,4.489)(2.1312,4.5234)(2.16,4.56)(2.1888,4.599)(2.2176,4.6409)
(2.2464,4.6858)(2.2752,4.7342)(2.304,4.7866)(2.3328,4.8435)(2.3616,4.9057)(2.3904,4.974)(2.4192,5.0495)
(2.448,5.1336)(2.4624,5.1794)(2.4768,5.228)(2.4912,5.2797)(2.5056,5.335)(2.52,5.3941)(2.5344,5.4576)(2.5488,5.5261)
(2.5632,5.6001)(2.5711,5.6447)
\thinlines 
\path(0,4.1295)(0.0288,4.1323)(0.0576,4.1351)(0.0864,4.138)(0.1152,4.141)(0.144,4.144)(0.1728,4.1471)
(0.2016,4.1502)(0.2304,4.1534)(0.2592,4.1566)(0.288,4.16)(0.3168,4.1633)(0.3456,4.1668)(0.3744,4.1703)
(0.4032,4.1739)(0.432,4.1776)(0.4608,4.1814)(0.4896,4.1852)(0.5184,4.1891)(0.5472,4.1932)(0.576,4.1973)
(0.6048,4.2015)(0.6336,4.2058)(0.6624,4.2101)(0.6912,4.2146)(0.72,4.2193)(0.7488,4.224)(0.7776,4.2288)
(0.8064,4.2338)(0.8352,4.2389)(0.864,4.2441)(0.8928,4.2494)(0.9216,4.2549)(0.9504,4.2606)(0.9792,4.2663)
(1.008,4.2723)(1.0368,4.2784)(1.0656,4.2847)(1.0944,4.2912)(1.1232,4.2978)(1.152,4.3047)(1.1808,4.3118)
(1.2096,4.3191)(1.2384,4.3266)(1.2672,4.3343)(1.296,4.3423)(1.3248,4.3506)(1.3536,4.3591)(1.3824,4.368)
(1.4112,4.3771)(1.44,4.3866)(1.4688,4.3964)(1.4976,4.4066)(1.5264,4.4171)(1.5552,4.4281)(1.584,4.4395)
(1.6128,4.4513)(1.6416,4.4636)(1.6704,4.4765)(1.6992,4.4899)(1.728,4.5039)(1.7568,4.5185)(1.7856,4.5338)
(1.8144,4.5498)(1.8432,4.5666)(1.872,4.5843)(1.9008,4.6029)(1.9296,4.6224)(1.9584,4.643)(1.9872,4.6648)
(2.016,4.6879)(2.0448,4.7123)(2.0736,4.7383)(2.1024,4.766)(2.1168,4.7805)(2.1312,4.7956)(2.1456,4.8111)
(2.16,4.8272)(2.1744,4.8439)(2.1888,4.8611)(2.2032,4.8791)(2.2176,4.8977)(2.232,4.917)(2.2464,4.9371)
(2.2608,4.9581)(2.2752,4.9799)(2.2896,5.0027)(2.304,5.0265)(2.3184,5.0513)(2.3328,5.0774)(2.3472,5.1046)
(2.3616,5.1333)(2.376,5.1634)(2.3904,5.1951)(2.4048,5.2285)(2.4192,5.2637)(2.4336,5.3011)(2.448,5.3407)
(2.4624,5.3828)(2.4696,5.4048)
\path(2.4696,5.4048)(2.4768,5.4276)(2.484,5.4511)(2.4912,5.4754)(2.4984,5.5006)(2.5056,5.5266)(2.5128,5.5537)
(2.52,5.5817)(2.5272,5.6108)(2.5344,5.641)(2.5352,5.6447)
\put(0.3224,4.4224){\makebox(0,0){\tiny $h=1$}}
\path(0.324,2.7997)(0.3312,2.7692)(0.3456,2.7358)(0.36,2.7118)(0.3744,2.692)(0.3888,2.6746)(0.4032,2.6589)
(0.432,2.6308)(0.4608,2.6058)(0.4896,2.5829)(0.5184,2.5615)(0.5472,2.5412)(0.576,2.5218)(0.6048,2.5032)
(0.6336,2.4851)(0.6624,2.4675)(0.6912,2.4503)(0.72,2.4334)(0.7488,2.4168)(0.7776,2.4004)(0.8064,2.3841)
(0.8352,2.368)(0.864,2.352)(0.8928,2.3361)(0.9216,2.3202)(0.9504,2.3043)(0.9792,2.2885)(1.008,2.2726)
(1.0368,2.2566)(1.0656,2.2406)(1.0944,2.2245)(1.1232,2.2083)(1.152,2.192)(1.1808,2.1756)(1.2096,2.159)
(1.2384,2.1422)(1.2672,2.1252)(1.296,2.108)(1.3248,2.0906)(1.3536,2.0729)(1.3824,2.055)(1.4112,2.0367)
(1.44,2.0182)(1.4688,1.9992)(1.4976,1.98)(1.5264,1.9603)(1.5552,1.9402)(1.584,1.9196)(1.6128,1.8986)(1.6416,1.877)
(1.6704,1.8548)(1.6992,1.8321)(1.728,1.8087)(1.7568,1.7846)(1.7856,1.7598)(1.8144,1.7342)(1.8432,1.7077)
(1.872,1.6803)(1.9008,1.6518)(1.9152,1.6372)(1.9296,1.6223)(1.944,1.6071)(1.9584,1.5916)(1.9728,1.5758)
(1.9872,1.5596)(2.0016,1.5431)(2.016,1.5262)(2.0304,1.5089)(2.0448,1.4912)(2.0592,1.4731)(2.0736,1.4545)
(2.088,1.4355)(2.1024,1.416)(2.1168,1.396)(2.1312,1.3755)(2.1456,1.3543)(2.16,1.3326)(2.1744,1.3103)(2.1888,1.2873)
(2.2032,1.2635)(2.2176,1.2391)(2.232,1.2138)(2.2464,1.1877)(2.2608,1.1607)(2.2752,1.1327)(2.2896,1.1037)
(2.304,1.0736)(2.3184,1.0424)(2.3328,1.0099)(2.3472,0.9761)(2.3616,0.9408)(2.376,0.904)(2.3904,0.8654)
(2.4048,0.8251)(2.4192,0.7827)(2.4264,0.7607)(2.4336,0.7381)(2.4408,0.715)(2.448,0.6912)(2.4552,0.6668)
\path(2.4552,0.6668)(2.4624,0.6416)(2.4696,0.6158)(2.4768,0.5892)(2.484,0.5618)(2.4912,0.5336)(2.4984,0.5044)
(2.5056,0.4744)(2.5128,0.4433)(2.52,0.4112)(2.5272,0.378)(2.5344,0.3436)(2.5416,0.3079)(2.5488,0.2709)
(2.556,0.2325)(2.5632,0.1925)(2.5704,0.1509)(2.5776,0.1076)(2.5848,0.0624)(2.592,0.0151)(2.5942,0)
\path(0.324,2.845)(0.3312,2.8755)(0.3456,2.9089)(0.36,2.9329)(0.3744,2.9527)(0.3888,2.9701)(0.4032,2.9858)
(0.432,3.0139)(0.4608,3.0389)(0.4896,3.0618)(0.5184,3.0832)(0.5472,3.1035)(0.576,3.1229)(0.6048,3.1415)
(0.6336,3.1596)(0.6624,3.1772)(0.6912,3.1944)(0.72,3.2113)(0.7488,3.228)(0.7776,3.2444)(0.8064,3.2606)
(0.8352,3.2767)(0.864,3.2927)(0.8928,3.3087)(0.9216,3.3246)(0.9504,3.3404)(0.9792,3.3563)(1.008,3.3722)
(1.0368,3.3881)(1.0656,3.4041)(1.0944,3.4202)(1.1232,3.4364)(1.152,3.4527)(1.1808,3.4691)(1.2096,3.4858)
(1.2384,3.5025)(1.2672,3.5195)(1.296,3.5367)(1.3248,3.5541)(1.3536,3.5718)(1.3824,3.5898)(1.4112,3.608)
(1.44,3.6266)(1.4688,3.6455)(1.4976,3.6648)(1.5264,3.6845)(1.5552,3.7046)(1.584,3.7251)(1.6128,3.7462)
(1.6416,3.7678)(1.6704,3.7899)(1.6992,3.8126)(1.728,3.836)(1.7568,3.8601)(1.7856,3.8849)(1.8144,3.9105)
(1.8432,3.937)(1.872,3.9645)(1.9008,3.9929)(1.9152,4.0075)(1.9296,4.0224)(1.944,4.0376)(1.9584,4.0531)
(1.9728,4.069)(1.9872,4.0852)(2.0016,4.1017)(2.016,4.1186)(2.0304,4.1359)(2.0448,4.1535)(2.0592,4.1716)
(2.0736,4.1902)(2.088,4.2092)(2.1024,4.2287)(2.1168,4.2487)(2.1312,4.2693)(2.1456,4.2904)(2.16,4.3121)
(2.1744,4.3345)(2.1888,4.3575)(2.2032,4.3812)(2.2176,4.4057)(2.232,4.431)(2.2464,4.4571)(2.2608,4.4841)
(2.2752,4.512)(2.2896,4.541)(2.304,4.5711)(2.3184,4.6023)(2.3328,4.6348)(2.3472,4.6687)(2.3616,4.7039)
(2.376,4.7408)(2.3904,4.7793)(2.4048,4.8197)(2.4192,4.8621)(2.4264,4.884)(2.4336,4.9066)(2.4408,4.9298)
(2.448,4.9535)(2.4552,4.978)
\path(2.4552,4.978)(2.4624,5.0031)(2.4696,5.0289)(2.4768,5.0555)(2.484,5.0829)(2.4912,5.1112)(2.4984,5.1403)
(2.5056,5.1704)(2.5128,5.2014)(2.52,5.2335)(2.5272,5.2668)(2.5344,5.3012)(2.5416,5.3368)(2.5488,5.3738)
(2.556,5.4123)(2.5632,5.4522)(2.5704,5.4938)(2.5776,5.5371)(2.5848,5.5824)(2.592,5.6296)(2.5942,5.6447)
\put(1.4224,3.0724){\makebox(0,0){\tiny $h=-2$}}
\put(0.5224,3.4224){\makebox(0,0){\tiny $h=-1$}}
\path(1.8234,2.7776)(1.8252,2.7474)(1.8288,2.7089)(1.8324,2.6802)(1.836,2.6563)(1.8432,2.6163)(1.8504,2.5823)
(1.8576,2.5522)(1.8648,2.5248)(1.872,2.4993)(1.8792,2.4753)(1.8864,2.4526)(1.9008,2.4099)(1.9152,2.3701)
(1.9296,2.3323)(1.944,2.2962)(1.9584,2.2613)(1.9728,2.2273)(1.9872,2.1942)(2.0016,2.1616)(2.016,2.1294)
(2.0304,2.0976)(2.0448,2.0661)(2.0592,2.0346)(2.0736,2.0033)(2.088,1.9719)(2.1024,1.9405)(2.1168,1.9089)
(2.1312,1.8771)(2.1456,1.8451)(2.16,1.8128)(2.1744,1.7801)(2.1888,1.7469)(2.2032,1.7133)(2.2176,1.6792)
(2.232,1.6444)(2.2464,1.609)(2.2608,1.5729)(2.2752,1.5359)(2.2896,1.4981)(2.304,1.4593)(2.3184,1.4194)
(2.3328,1.3784)(2.3472,1.3362)(2.3544,1.3146)(2.3616,1.2926)(2.3688,1.2702)(2.376,1.2475)(2.3832,1.2243)
(2.3904,1.2008)(2.3976,1.1767)(2.4048,1.1523)(2.412,1.1273)(2.4192,1.1018)(2.4264,1.0758)(2.4336,1.0492)
(2.4408,1.0221)(2.448,0.9943)(2.4552,0.9659)(2.4624,0.9368)(2.4696,0.9069)(2.4768,0.8763)(2.484,0.845)
(2.4912,0.8127)(2.4984,0.7796)(2.5056,0.7456)(2.5128,0.7105)(2.52,0.6744)(2.5272,0.6372)(2.5344,0.5988)
(2.5416,0.5591)(2.5488,0.518)(2.556,0.4756)(2.5632,0.4315)(2.5704,0.3859)(2.5776,0.3384)(2.5848,0.2891)
(2.5884,0.2637)(2.592,0.2377)(2.5956,0.2112)(2.5992,0.1841)(2.6028,0.1564)(2.6064,0.1281)(2.61,0.0991)
(2.6136,0.0695)(2.6172,0.0391)(2.6208,0.008)(2.6217,0)
\put(3.4224,5.5224){\makebox(0,0){\small $\dot r$}}
\put(5.3224,3.9224){\makebox(0,0){\small $du_+$}}
\path(1.8234,2.8671)(1.8252,2.8973)(1.8288,2.9359)(1.8324,2.9645)(1.836,2.9884)(1.8432,3.0285)(1.8504,3.0624)
(1.8576,3.0925)(1.8648,3.12)(1.872,3.1454)(1.8792,3.1694)(1.8864,3.1922)(1.9008,3.2348)(1.9152,3.2747)
(1.9296,3.3124)(1.944,3.3486)(1.9584,3.3835)(1.9728,3.4174)(1.9872,3.4506)(2.0016,3.4832)(2.016,3.5153)
(2.0304,3.5471)(2.0448,3.5787)(2.0592,3.6101)(2.0736,3.6415)(2.088,3.6728)(2.1024,3.7043)(2.1168,3.7359)
(2.1312,3.7676)(2.1456,3.7997)(2.16,3.832)(2.1744,3.8647)(2.1888,3.8978)(2.2032,3.9314)(2.2176,3.9656)
(2.232,4.0003)(2.2464,4.0357)(2.2608,4.0719)(2.2752,4.1088)(2.2896,4.1467)(2.304,4.1855)(2.3184,4.2253)
(2.3328,4.2663)(2.3472,4.3086)(2.3544,4.3302)(2.3616,4.3522)(2.3688,4.3745)(2.376,4.3973)(2.3832,4.4204)
(2.3904,4.444)(2.3976,4.468)(2.4048,4.4925)(2.412,4.5174)(2.4192,4.5429)(2.4264,4.5689)(2.4336,4.5955)
(2.4408,4.6227)(2.448,4.6505)(2.4552,4.6789)(2.4624,4.708)(2.4696,4.7378)(2.4768,4.7684)(2.484,4.7998)
(2.4912,4.832)(2.4984,4.8651)(2.5056,4.8992)(2.5128,4.9342)(2.52,4.9703)(2.5272,5.0076)(2.5344,5.046)
(2.5416,5.0857)(2.5488,5.1267)(2.556,5.1692)(2.5632,5.2132)(2.5704,5.2589)(2.5776,5.3063)(2.5848,5.3556)
(2.5884,5.3811)(2.592,5.407)(2.5956,5.4336)(2.5992,5.4607)(2.6028,5.4884)(2.6064,5.5167)(2.61,5.5456)
(2.6136,5.5753)(2.6172,5.6056)(2.6208,5.6367)(2.6217,5.6447)
\put(5.3224,1.7224){\makebox(0,0){\small $du_-$}}
\end{picture}%

%% file: SymKepler.tex
\unitlength 1cm
\begin{picture}(6,6)(-0.5,-0.5)%
\color{black}\thinlines 
\put(2.5,2.5){\makebox(0,0){\tiny$\bullet$}}
\path(2.5,0)(2.5,5)
\path(5,2.5)(0,2.5)
\path(1.25,0)(3.75,5)
\path(5,1.25)(0,3.75)
\path(5,3.75)(0,1.25)
\path(3.75,0)(1.25,5)
\path(2.5,0.51)(2.5,0.5)
\path(2.5667,0.6155)(2.5,0.5)(2.4333,0.6155)
\path(0.51,2.5)(0.5,2.5)
\path(0.6155,2.4333)(0.5,2.5)(0.6155,2.5667)
\path(1.5045,0.5089)(1.5,0.5)
\path(1.6113,0.5735)(1.5,0.5)(1.492,0.6331)
\path(4.4911,1.5045)(4.5,1.5)
\path(4.4265,1.6113)(4.5,1.5)(4.3669,1.492)
\path(4.4911,3.4955)(4.5,3.5)
\path(4.3669,3.508)(4.5,3.5)(4.4265,3.3887)
\path(3.4955,0.5089)(3.5,0.5)
\path(3.508,0.6331)(3.5,0.5)(3.3887,0.5735)
\path(1.5045,4.4911)(1.5,4.5)
\path(1.492,4.3669)(1.5,4.5)(1.6113,4.4265)
\path(0.5089,1.5045)(0.5,1.5)
\path(0.6331,1.492)(0.5,1.5)(0.5735,1.6113)
\path(0.5089,3.4955)(0.5,3.5)
\path(0.5735,3.3887)(0.5,3.5)(0.6331,3.508)
\path(3.4955,4.4911)(3.5,4.5)
\path(3.3887,4.4265)(3.5,4.5)(3.508,4.3669)
\path(4.49,2.5)(4.5,2.5)
\path(4.3845,2.5667)(4.5,2.5)(4.3845,2.4333)
\path(2.5,4.49)(2.5,4.5)
\path(2.4333,4.3845)(2.5,4.5)(2.5667,4.3845)
\end{picture}%

%% file: BusKepler.tex
\unitlength 1cm
\begin{picture}(6,6)(-0.5,-0.5)%
\color{black}\thicklines 
\dashline{0.15}(0,2.5)(2.5,2.5)
\thinlines 
\put(2.5,2.5){\makebox(0,0){\tiny$\bullet$}}
\path(2.5,2.5)(5,2.5)
\path(4.9988,0.051)(4.9366,0.0765)(4.8751,0.102)(4.8143,0.1276)(4.7541,0.1531)(4.6945,0.1786)(4.6356,0.2041)
(4.5774,0.2296)(4.5198,0.2551)(4.4628,0.2806)(4.4065,0.3061)(4.3509,0.3316)(4.2959,0.3571)(4.2416,0.3827)
(4.1879,0.4082)(4.1349,0.4337)(4.0825,0.4592)(4.0307,0.4847)(3.9796,0.5102)(3.9292,0.5357)(3.8794,0.5612)
(3.8303,0.5867)(3.7818,0.6122)(3.734,0.6378)(3.6868,0.6633)(3.6403,0.6888)(3.5944,0.7143)(3.5492,0.7398)
(3.5046,0.7653)(3.4607,0.7908)(3.4174,0.8163)(3.3748,0.8418)(3.3328,0.8673)(3.2915,0.8929)(3.2508,0.9184)
(3.2108,0.9439)(3.1714,0.9694)(3.1327,0.9949)(3.0946,1.0204)(3.0572,1.0459)(3.0204,1.0714)(2.9843,1.0969)
(2.9488,1.1224)(2.914,1.148)(2.8798,1.1735)(2.8463,1.199)(2.8135,1.2245)(2.7813,1.25)(2.7497,1.2755)(2.7188,1.301)
(2.6885,1.3265)(2.6589,1.352)(2.6299,1.3776)(2.6016,1.4031)(2.574,1.4286)(2.547,1.4541)(2.5206,1.4796)
(2.4699,1.5306)(2.4217,1.5816)(2.3761,1.6327)(2.3332,1.6837)(2.2928,1.7347)(2.2551,1.7857)(2.22,1.8367)
(2.1874,1.8878)(2.1575,1.9388)(2.1302,1.9898)(2.1054,2.0408)(2.0833,2.0918)(2.0638,2.1429)(2.0469,2.1939)
(2.0325,2.2449)(2.0208,2.2959)(2.0117,2.3469)(2.0052,2.398)(2.0013,2.449)(2,2.5)(2.0013,2.551)(2.0052,2.602)
(2.0117,2.6531)(2.0208,2.7041)(2.0325,2.7551)(2.0469,2.8061)(2.0638,2.8571)(2.0833,2.9082)(2.1054,2.9592)
(2.1302,3.0102)(2.1575,3.0612)(2.1874,3.1122)(2.22,3.1633)(2.2551,3.2143)(2.2928,3.2653)(2.3332,3.3163)
(2.3761,3.3673)(2.4217,3.4184)(2.4699,3.4694)(2.5206,3.5204)(2.547,3.5459)(2.574,3.5714)(2.6016,3.5969)
\path(2.6016,3.5969)(2.6299,3.6224)(2.6589,3.648)(2.6885,3.6735)(2.7188,3.699)(2.7497,3.7245)(2.7812,3.75)
(2.8135,3.7755)(2.8463,3.801)(2.8798,3.8265)(2.914,3.852)(2.9488,3.8776)(2.9843,3.9031)(3.0204,3.9286)
(3.0572,3.9541)(3.0946,3.9796)(3.1327,4.0051)(3.1714,4.0306)(3.2108,4.0561)(3.2508,4.0816)(3.2915,4.1071)
(3.3328,4.1327)(3.3748,4.1582)(3.4174,4.1837)(3.4607,4.2092)(3.5046,4.2347)(3.5492,4.2602)(3.5944,4.2857)
(3.6403,4.3112)(3.6868,4.3367)(3.734,4.3622)(3.7818,4.3878)(3.8303,4.4133)(3.8794,4.4388)(3.9292,4.4643)
(3.9796,4.4898)(4.0307,4.5153)(4.0825,4.5408)(4.1349,4.5663)(4.1879,4.5918)(4.2416,4.6173)(4.2959,4.6429)
(4.3509,4.6684)(4.4065,4.6939)(4.4628,4.7194)(4.5198,4.7449)(4.5774,4.7704)(4.6356,4.7959)(4.6945,4.8214)
(4.7541,4.8469)(4.8143,4.8724)(4.8751,4.898)(4.9366,4.9235)(5,4.9495)
\path(2.2812,0)(2.2497,0.051)(2.2188,0.102)(2.1885,0.1531)(2.1589,0.2041)(2.1299,0.2551)(2.1016,0.3061)
(2.074,0.3571)(2.047,0.4082)(2.0206,0.4592)(1.9949,0.5102)(1.9699,0.5612)(1.9455,0.6122)(1.9217,0.6633)
(1.8986,0.7143)(1.8761,0.7653)(1.8543,0.8163)(1.8332,0.8673)(1.8127,0.9184)(1.7928,0.9694)(1.7736,1.0204)
(1.7551,1.0714)(1.7372,1.1224)(1.72,1.1735)(1.7034,1.2245)(1.6874,1.2755)(1.6721,1.3265)(1.6575,1.3776)
(1.6435,1.4286)(1.6302,1.4796)(1.6175,1.5306)(1.6054,1.5816)(1.594,1.6327)(1.5833,1.6837)(1.5732,1.7347)
(1.5638,1.7857)(1.555,1.8367)(1.5469,1.8878)(1.5394,1.9388)(1.5325,1.9898)(1.5264,2.0408)(1.5208,2.0918)
(1.5159,2.1429)(1.5117,2.1939)(1.5081,2.2449)(1.5052,2.2959)(1.5029,2.3469)(1.5013,2.398)(1.5003,2.449)
(1.5,2.5)(1.5003,2.551)(1.5013,2.602)(1.5029,2.6531)(1.5052,2.7041)(1.5081,2.7551)(1.5117,2.8061)(1.5159,2.8571)
(1.5208,2.9082)(1.5264,2.9592)(1.5325,3.0102)(1.5394,3.0612)(1.5469,3.1122)(1.555,3.1633)(1.5638,3.2143)
(1.5732,3.2653)(1.5833,3.3163)(1.594,3.3673)(1.6054,3.4184)(1.6175,3.4694)(1.6302,3.5204)(1.6435,3.5714)
(1.6575,3.6224)(1.6721,3.6735)(1.6874,3.7245)(1.7034,3.7755)(1.72,3.8265)(1.7372,3.8776)(1.7551,3.9286)
(1.7736,3.9796)(1.7928,4.0306)(1.8127,4.0816)(1.8332,4.1327)(1.8543,4.1837)(1.8761,4.2347)(1.8986,4.2857)
(1.9217,4.3367)(1.9455,4.3878)(1.9699,4.4388)(1.9949,4.4898)(2.0206,4.5408)(2.047,4.5918)(2.074,4.6429)
(2.1016,4.6939)(2.1299,4.7449)(2.1589,4.7959)(2.1885,4.8469)(2.2188,4.898)(2.2497,4.949)(2.2812,5)
\path(0.6953,0)(0.6874,0.051)(0.6797,0.102)(0.6721,0.1531)(0.6647,0.2041)(0.6575,0.2551)(0.6504,0.3061)
(0.6435,0.3571)(0.6367,0.4082)(0.6302,0.4592)(0.6237,0.5102)(0.6175,0.5612)(0.6114,0.6122)(0.6054,0.6633)
(0.5996,0.7143)(0.594,0.7653)(0.5886,0.8163)(0.5833,0.8673)(0.5782,0.9184)(0.5732,0.9694)(0.5684,1.0204)
(0.5638,1.0714)(0.5593,1.1224)(0.555,1.1735)(0.5508,1.2245)(0.5469,1.2755)(0.543,1.3265)(0.5394,1.3776)
(0.5359,1.4286)(0.5325,1.4796)(0.5294,1.5306)(0.5264,1.5816)(0.5235,1.6327)(0.5208,1.6837)(0.5183,1.7347)
(0.5159,1.7857)(0.5137,1.8367)(0.5117,1.8878)(0.5098,1.9388)(0.5081,1.9898)(0.5066,2.0408)(0.5052,2.0918)
(0.504,2.1429)(0.5029,2.1939)(0.502,2.2449)(0.5013,2.2959)(0.5007,2.3469)(0.5003,2.398)(0.5001,2.449)
(0.5001,2.551)(0.5003,2.602)(0.5007,2.6531)(0.5013,2.7041)(0.502,2.7551)(0.5029,2.8061)(0.504,2.8571)
(0.5052,2.9082)(0.5066,2.9592)(0.5081,3.0102)(0.5098,3.0612)(0.5117,3.1122)(0.5137,3.1633)(0.5159,3.2143)
(0.5183,3.2653)(0.5208,3.3163)(0.5235,3.3673)(0.5264,3.4184)(0.5294,3.4694)(0.5325,3.5204)(0.5359,3.5714)
(0.5394,3.6224)(0.543,3.6735)(0.5469,3.7245)(0.5508,3.7755)(0.555,3.8265)(0.5593,3.8776)(0.5638,3.9286)
(0.5684,3.9796)(0.5732,4.0306)(0.5782,4.0816)(0.5833,4.1327)(0.5886,4.1837)(0.594,4.2347)(0.5996,4.2857)
(0.6054,4.3367)(0.6114,4.3878)(0.6175,4.4388)(0.6237,4.4898)(0.6302,4.5408)(0.6367,4.5918)(0.6435,4.6429)
(0.6504,4.6939)(0.6575,4.7449)(0.6647,4.7959)(0.6721,4.8469)(0.6797,4.898)(0.6874,4.949)(0.6953,5)
\path(5,1.3274)(4.9445,1.3393)(4.8856,1.352)(4.8274,1.3648)(4.7698,1.3776)(4.7128,1.3903)(4.6565,1.4031)
(4.6009,1.4158)(4.5459,1.4286)(4.4916,1.4413)(4.4379,1.4541)(4.3849,1.4668)(4.3325,1.4796)(4.2807,1.4923)
(4.2296,1.5051)(4.1792,1.5179)(4.1294,1.5306)(4.0803,1.5434)(4.0318,1.5561)(3.984,1.5689)(3.9368,1.5816)
(3.8903,1.5944)(3.8444,1.6071)(3.7992,1.6199)(3.7546,1.6327)(3.7107,1.6454)(3.6674,1.6582)(3.6248,1.6709)
(3.5828,1.6837)(3.5415,1.6964)(3.5008,1.7092)(3.4608,1.7219)(3.4214,1.7347)(3.3827,1.7474)(3.3446,1.7602)
(3.2704,1.7857)(3.1988,1.8112)(3.1298,1.8367)(3.0635,1.8622)(2.9997,1.8878)(2.9385,1.9133)(2.8799,1.9388)
(2.824,1.9643)(2.7706,1.9898)(2.7199,2.0153)(2.6717,2.0408)(2.6261,2.0663)(2.5832,2.0918)(2.5428,2.1173)
(2.5051,2.1429)(2.47,2.1684)(2.4374,2.1939)(2.4075,2.2194)(2.3802,2.2449)(2.3333,2.2959)(2.2969,2.3469)
(2.2708,2.398)(2.2552,2.449)(2.25,2.5)(2.2552,2.551)(2.2708,2.602)(2.2969,2.6531)(2.3333,2.7041)(2.3802,2.7551)
(2.4075,2.7806)(2.4374,2.8061)(2.47,2.8316)(2.5051,2.8571)(2.5428,2.8827)(2.5832,2.9082)(2.6261,2.9337)
(2.6717,2.9592)(2.7199,2.9847)(2.7706,3.0102)(2.824,3.0357)(2.8799,3.0612)(2.9385,3.0867)(2.9997,3.1122)
(3.0635,3.1378)(3.1298,3.1633)(3.1988,3.1888)(3.2704,3.2143)(3.3446,3.2398)(3.3827,3.2526)(3.4214,3.2653)
(3.4608,3.2781)(3.5008,3.2908)(3.5415,3.3036)(3.5828,3.3163)(3.6248,3.3291)(3.6674,3.3418)(3.7107,3.3546)
(3.7546,3.3673)(3.7992,3.3801)(3.8444,3.3929)(3.8903,3.4056)(3.9368,3.4184)(3.984,3.4311)(4.0318,3.4439)
(4.0803,3.4566)
\path(4.0803,3.4566)(4.1294,3.4694)(4.1792,3.4821)(4.2296,3.4949)(4.2807,3.5077)(4.3325,3.5204)(4.3849,3.5332)
(4.4379,3.5459)(4.4916,3.5587)(4.5459,3.5714)(4.6009,3.5842)(4.6565,3.5969)(4.7128,3.6097)(4.7698,3.6224)
(4.8274,3.6352)(4.8856,3.648)(4.9445,3.6607)(5,3.6726)
\path(5,1.9272)(4.9524,1.9324)(4.8948,1.9388)(4.8378,1.9452)(4.7815,1.9515)(4.7259,1.9579)(4.6709,1.9643)
(4.6166,1.9707)(4.5629,1.977)(4.5099,1.9834)(4.4575,1.9898)(4.4057,1.9962)(4.3546,2.0026)(4.3042,2.0089)
(4.2544,2.0153)(4.2053,2.0217)(4.1568,2.0281)(4.109,2.0344)(4.0618,2.0408)(4.0153,2.0472)(3.9694,2.0536)
(3.9242,2.0599)(3.8796,2.0663)(3.7924,2.0791)(3.7078,2.0918)(3.6258,2.1046)(3.5464,2.1173)(3.4696,2.1301)
(3.3954,2.1429)(3.3238,2.1556)(3.2548,2.1684)(3.1885,2.1811)(3.1247,2.1939)(3.0635,2.2066)(3.0049,2.2194)
(2.949,2.2321)(2.8956,2.2449)(2.8449,2.2577)(2.7967,2.2704)(2.7511,2.2832)(2.7082,2.2959)(2.6678,2.3087)
(2.6301,2.3214)(2.5624,2.3469)(2.5052,2.3724)(2.4583,2.398)(2.4219,2.4235)(2.3958,2.449)(2.375,2.5)(2.3958,2.551)
(2.4219,2.5765)(2.4583,2.602)(2.5052,2.6276)(2.5624,2.6531)(2.6301,2.6786)(2.6678,2.6913)(2.7082,2.7041)
(2.7511,2.7168)(2.7967,2.7296)(2.8449,2.7423)(2.8956,2.7551)(2.949,2.7679)(3.0049,2.7806)(3.0635,2.7934)
(3.1247,2.8061)(3.1885,2.8189)(3.2548,2.8316)(3.3238,2.8444)(3.3954,2.8571)(3.4696,2.8699)(3.5464,2.8827)
(3.6258,2.8954)(3.7078,2.9082)(3.7924,2.9209)(3.8796,2.9337)(3.9242,2.9401)(3.9694,2.9464)(4.0153,2.9528)
(4.0618,2.9592)(4.109,2.9656)(4.1568,2.9719)(4.2053,2.9783)(4.2544,2.9847)(4.3042,2.9911)(4.3546,2.9974)
(4.4057,3.0038)(4.4575,3.0102)(4.5099,3.0166)(4.5629,3.023)(4.6166,3.0293)(4.6709,3.0357)(4.7259,3.0421)
(4.7815,3.0485)(4.8378,3.0548)(4.8948,3.0612)(4.9524,3.0676)(5,3.0728)
\put(0.5,2.5){\makebox(0,0){\tiny$\bullet$}}
\put(1.5,2.5){\makebox(0,0){\tiny$\bullet$}}
\put(2,2.5){\makebox(0,0){\tiny$\bullet$}}
\put(2.25,2.5){\makebox(0,0){\tiny$\bullet$}}
\put(2.375,2.5){\makebox(0,0){\tiny$\bullet$}}
\path(3.527,0.7525)(3.5357,0.7475)
\path(3.4686,0.8627)(3.5357,0.7475)(3.4024,0.747)
\path(2.0149,0.4705)(2.0194,0.4616)
\path(2.027,0.5947)(2.0194,0.4616)(1.9079,0.5347)
\path(0.6241,0.507)(0.6254,0.4971)
\path(0.6771,0.62)(0.6254,0.4971)(0.5448,0.6033)
\path(4.2669,1.4958)(4.2766,1.4934)
\path(4.1807,1.586)(4.2766,1.4934)(4.1484,1.4567)
\path(4.4476,1.991)(4.4575,1.9898)
\path(4.351,2.0699)(4.4575,1.9898)(4.3349,1.9376)
\path(4.49,2.5)(4.5,2.5)
\path(4.3845,2.5667)(4.5,2.5)(4.3845,2.4333)
\path(3.5282,4.2482)(3.5369,4.2532)
\path(3.4035,4.2537)(3.5369,4.2532)(3.4698,4.138)
\path(2.0149,4.5295)(2.0194,4.5384)
\path(1.9079,4.4653)(2.0194,4.5384)(2.027,4.4053)
\path(0.6241,4.493)(0.6254,4.5029)
\path(0.5448,4.3967)(0.6254,4.5029)(0.6771,4.38)
\path(4.2669,3.5042)(4.2766,3.5066)
\path(4.1484,3.5433)(4.2766,3.5066)(4.1807,3.414)
\path(4.4476,3.009)(4.4575,3.0102)
\path(4.3349,3.0624)(4.4575,3.0102)(4.351,2.9301)
\end{picture}%